\documentclass[11pt,a4paper,twoside]{amsart}
\usepackage[activeacute,english]{babel}
\usepackage[latin1]{inputenc}
\usepackage{amsfonts}
\usepackage{amsmath,amsthm}
\usepackage{amssymb}
\usepackage{graphics}

\newcommand{\A}{{\mathcal A}}

\newcommand{\CC}{{\mathcal C}}

\newcommand{\FF}{{\mathcal F}}

\newcommand{\HH}{{\mathcal H}}

\newcommand{\LL}{{\mathcal L}}

\newcommand{\OO}{{\mathcal O}}

\newcommand{\SSS}{{\mathcal S}}
\newcommand{\TT}{{\mathcal T}}
\newcommand{\VV}{{\mathcal V}}

\newcommand{\C}{{\mathbb C}}
\newcommand{\N}{{\mathbb N}}
\newcommand{\R}{{\mathbb R}}
\newcommand{\Z}{{\mathbb Z}}

\newcommand{\wit}{\widetilde}
\newcommand{\wih}{\widehat}

\newcommand{\supp}{{\operatorname{supp}}}

\newcommand{\Lip}{{\operatorname{Lip}}}

\newtheorem{teo}{Theorem}[section]

\newtheorem{lema}[teo]{Lemma}
\newtheorem{coro}[teo]{Corollary}

\newtheorem{defi}[teo]{Definition}

\newtheorem{remarko}[teo]{{\em Remark}}
{\theoremstyle{remark} }
\allowdisplaybreaks[1]

\title[Variation for singular integrals on Lipschitz graphs]
{Variation for singular integrals on\\Lipschitz graphs: $L^p$ and endpoint estimates}
\author[A. MAS]{ALBERT MAS}
\date{September, 2011}
\thanks{The author is partially supported by grants AP2006-02416 (FPU program, Spain), MTM2010-16232 (Spain), and 2009SGR-000420 (Generalitat de Catalunya, Spain).}
\address{Departamento de Matem\'aticas, Universidad del Pa\'is Vasco, 48080 Bilbao (Spain)} \email{amasblesa@gmail.com}

\begin{document}

\begin{abstract}
Let $1\leq n<d$ be integers and let $\mu$ denote the $n$-dimensional Hausdorff measure restricted to an $n$-dimensional Lipschitz graph in $\R^d$ with slope strictly less than $1$. For $\rho>2$, we prove that the $\rho$-variation and oscillation for Calder\'{o}n-Zygmund singular integrals with odd kernel are bounded operators in $L^{p}(\mu)$ for $1<p<\infty$, from $L^1(\mu)$ to $L^{1,\infty}(\mu)$, and from $L^\infty(\mu)$ to $BMO(\mu)$. Concerning the first endpoint estimate, we actually show that such operators are bounded
from the space of finite complex Radon measures in $\R^d$ to $L^{1,\infty}(\mu)$. 
\end{abstract}
\maketitle
\section{Introduction}
The $\rho$-variation and oscillation for martingales and some
families of operators have been studied in many recent papers on
probability, ergodic theory, and harmonic analysis (see
\cite{Lepingle}, \cite{Bourgain}, \cite{JKRW-ergodic},
\cite{CJRW-Hilbert}, \cite{JSW}, \cite{Lacey}, and
\cite{OSTTW}, for example). In this paper we continue the study developed in \cite{MT1} and \cite{MT2} about the $\rho$-variation and oscillation for Calder\'{o}n-Zygmund singular integral operators with odd kernel defined on measures different form the Lebesgue measure. More precisely, we are concerned with variational $L^p$ ($1<p<\infty$) and endpoint estimates for such singular integral operators defined on Lipschitz graphs and with respect to the Hausdorff measure.
 
Throughout the paper $1\leq n<d$ denote two fixed integers. By an $n$-dimensional Lipschitz graph $\Gamma\subset\R^d$ we mean any translation and rotation of a set of the type $$\{x\in\R^d\,:\,x=(y,\A(y)),\, y\in\R^n\},$$ where $\A:\R^n\to\R^{d-n}$ is some Lipschitz function with Lipschitz constant $\Lip(\A)$. We say that $\Lip(\A)$ is the {\em slope} of $\Gamma$. 

Given $1\leq n<d$ integers, $\epsilon>0$, and a Radon measure $\mu$ in $\R^d$, we consider
\begin{equation}\label{SIO}
\begin{split}
T_{\epsilon}\mu(x) := \int_{|x-y|>\epsilon}K(x-y)\,d\mu(y),\quad\text{ for }x\in\R^d,
\end{split}
\end{equation}
where the kernel $K:\R^d\setminus\{0\}\to\C$ satisfies
\begin{equation}\label{4eq333}
|K(x)|\leq \frac{C}{|x|^{n}},\quad|\partial_{x_i}K(x)|\leq
\frac{C}{|x|^{n+1}}\quad\text{and}\quad|\partial_{x_i}\partial_{x_j}K(x)|\leq\frac{C}{|x|^{n+2}},
\end{equation}
for all $1\leq i,j\leq d$ and $x=(x_1,\ldots,x_d)\in\R^d\setminus\{0\}$, $C>0$ is some constant, and moreover $K(-x)=-K(x)$ for
all $x\neq0$ (i.e. $K$ is odd). We set $\TT\mu:=\{T_\epsilon\mu\}_{\epsilon>0}$, and given 
$f\in L^1(\mu)$, we also set $T_\epsilon^\mu f:=T_\epsilon(f\mu)$, $T^\mu_* f(x):=\sup_{\epsilon>0}|T_\epsilon^\mu f(x)|$, and $\TT^\mu f:=\{T_\epsilon^\mu f\}_{\epsilon>0}$. The well-known Cauchy and $n$-dimensional Riesz transforms are two very important examples of such Calder\'{o}n-Zygmund singular integral operators, and they correspond to the kernels $K(x)=1/x$ for $x\in\C\setminus\{0\}$ and $K(x)=x/|x|^{n+1}$ for $x\in\R^d\setminus\{0\}$ respectively (to be precise, we should consider the scalar components $x_i/|x|^{n+1}$).

\begin{defi}[$\rho$-variation and oscillation]
Let $\mathcal{F}:=\{F_\epsilon\}_{\epsilon>0}$ be a family of
functions defined on $\R^d$. Given  $\rho>0$, the $\rho$-{\em
variation} of $\FF$ at $x\in\R^d$ is defined by
\begin{equation*}
\VV_{\rho}(\FF)(x):=\sup_{\{\epsilon_{m}\}}\bigg(\sum_{m\in\Z}
|F_{\epsilon_{m+1}}(x)-F_{\epsilon_{m}}(x)|^{\rho}\bigg)^{1/\rho},
\end{equation*}
where the pointwise supremum is taken over all decreasing
sequences $\{\epsilon_{m}\}_{m\in\Z}\subset(0,\infty)$. Fix a decreasing
sequence $\{r_{m}\}_{m\in\Z}\subset(0,\infty)$. The {\em oscillation} of $\FF$
at $x\in\R^d$ is defined by
\begin{equation*}
\OO(\FF)(x):=\sup_{\{\epsilon_m\},\{\delta_{m}\}}\bigg(\sum_{m\in\Z}
|F_{\epsilon_m}(x)-F_{\delta_m}(x)|^{2}\bigg)^{1/2},
\end{equation*}
where the pointwise supremum is taken over all sequences
$\{\epsilon_m\}_{m\in\Z}$ and $\{\delta_{m}\}_{m\in\Z}$
such that $r_{m+1}\leq\epsilon_{m}\leq\delta_{m}\leq r_{m}$ for all
$m\in\Z$.
\end{defi}

Given a Radon measure $\mu$ in $\R^d$, $f\in L^1(\mu)$, and $x\in\R^d$, we will deal with
\begin{equation*}
\begin{split}
(\VV_{\rho}\circ\TT)\mu(x):=\VV_{\rho}(\TT\mu)(x),\quad
&\text{and}\quad
(\VV_{\rho}\circ\TT^\mu)f(x):=\VV_{\rho}(\TT^\mu f)(x),\\ (\OO\circ\TT)\mu(x):=\OO(\TT\mu)(x),\quad&\text{and}\quad
(\OO\circ\TT^\mu)f(x):=\OO(\TT^\mu f)(x).
\end{split}
\end{equation*}
For a Borel set $E\subset\R^d$, we denote by $\HH^n_E$ the $n$-dimensional Hausdorff measure resticted to $E$. The following result is a direct consequence of \cite[Theorem 1.3]{MT2}.
\begin{teo}\label{4main theorem lip}
Let $\rho>2$. Let $\Gamma\subset\R^d$ be an $n$-dimensional Lipschitz graph and set $\mu:=\HH^{n}_{\Gamma}$. Then, 
$\VV_{\rho}\circ\TT^{\mu}$ and
$\OO\circ\TT^{\mu}$ are bounded operators
in $L^2(\mu)$. The norms of these operators are bounded by some constants depending only on $n$, $d$, $K$, and the slope of $\Gamma$ (and on $\rho$ in the case of $\VV_{\rho}\circ\TT^{\mu}$). In particular, the norm of $\OO\circ\TT^{\mu}$ is bounded independently of the sequence that defines $\OO$.
\end{teo}
Actually, from \cite[Theorem 1.3]{MT2} one has that Theorem \ref{4main theorem lip} holds whenever $\mu$ is an $n$-dimensional AD regular uniformly $n$-rectifiable measure in $\R^d$ (see \cite[Part I]{DS2} for the precise definitions of AD regularity and uniform rectifiability). Let us just mention that this latter assumptions on $\mu$ are some geometry-measure theoretic properties of homogeneity and of quantitative rectifiability which are trivially satisfied by Lipschitz graphs. Furthermore, in \cite{MT1} it is also proved that, if $\mu=\HH^n_\Gamma$ for some $n$-dimensional Lipschitz graph $\Gamma\subset\R^d$, $\varphi\in\CC^\infty(\R)$ is some fixed function such that $\chi_{[2,\infty)}\leq\varphi\leq\chi_{[1/2,\infty)}$, 
\begin{equation}\label{smooth SIO}
\begin{split}
T_{\varphi_\epsilon}^\mu f(x) := \int\varphi(|x-y|/\epsilon)K(x-y)f(y)\,d\mu(y)\quad\text{ for }x\in\R^d\text{ and }f\in L^1(\mu),
\end{split}
\end{equation}
and $\TT_\varphi^{\mu}:=\{T_{\varphi_\epsilon}^{\mu}\}_{\epsilon>0}$, then the operators $\VV_{\rho}\circ\TT_\varphi^{\mu}$ and $\OO\circ\TT_\varphi^{\mu}$ are bounded 
\begin{itemize}
\item[$(a)$] in $L^p(\mu)$ for all $1<p<\infty$,
\item[$(b)$] from $L^1(\mu)$ to $L^{1,\infty}(\mu)$, and 
\item[$(c)$] from $L^\infty(\mu)$ to $BMO(\mu)$ (see Section \ref{teorema Lp no suau s2} for the precise definition of $BMO(\mu)$).
\end{itemize}

Usually, we refer to $\TT^{\mu}$ as {\em the family of rough truncations of the singular integral operator with kernel $K$ and with respect to $\mu$}, and we refer to $\TT_\varphi^{\mu}$ as {\em the family of smooth truncations} of the same operator.

The following theorem is one of the main results of this paper. Roughly speaking, under an extra assumption on the slope of the Lipschitz graph, it improves Theorem \ref{4main theorem lip} and extends the estimates $(a)$, $(b)$, and $(c)$ above to rough truncations.

\begin{teo}\label{4unif rectif teo3b}
Let $\rho>2$. Let $\Gamma\subset\R^d$ be an $n$-dimensional Lipschitz graph with slope strictly less than $1$ and set $\mu:=\HH^{n}_{\Gamma}$. Then, $\VV_{\rho}\circ\TT^{\mu}$ and
$\OO\circ\TT^{\mu}$ are bounded operators
\begin{itemize}
\item[$(a)$] in $L^p(\mu)$ for all $1<p<\infty$,
\item[$(b)$] from $L^1(\mu)$ to $L^{1,\infty}(\mu)$, and 
\item[$(c)$] from $L^\infty(\mu)$ to $BMO(\mu)$,
\end{itemize}
The norms of these operators are bounded by some constants depending only on $n$, $d$, $K$, the slope of $\Gamma$, on $\rho$ in the case of $\VV_{\rho}\circ\TT^{\mu}$, and on $p$ in the case of $(a)$. In particular, the norm of $\OO\circ\TT^{\mu}$ is bounded independently of the sequence that defines $\OO$.
\end{teo}
This theorem generalizes the results in \cite{CJRW-singular integrals} for the class of kernels given by (\ref{4eq333}) and, in this sense, it is a natural continuation of the study of variational inequalities for Calder\'on-Zygmund singular integral operators.

As we pointed out above, Theorem \ref{4unif rectif teo3b} was already known for the family $\TT^{\mu}_\varphi$, but the case of rough truncations requires much more work and detail on the estimates due to the lack of regularity on the truncations. Moreover, \cite[Theorem 1.3]{MT2} (and so Theorem \ref{4main theorem lip}) where obtained using the so-called {\em corona decomposition} (see \cite[Chapter 3 of Part I]{DS2}), which is a useful tool to deal with  $L^2$ estimates. However, it is very difficult to adapt that techniques to deal with $L^p$ estimates for $p\neq2$. Thus, Theorem \ref{4unif rectif teo3b} does not follow from the variational $L^p$ estimates for $\TT^{\mu}_\varphi$, nor by a simple modification of the proof of Theorem \ref{4main theorem lip}, it requires a more careful and deeper study.

We denote by $M(\R^d)$ the space of finite complex Radon measures on $\R^d$ equipped with the norm given by the variation of measures.
The other main result of this paper is the following theorem, which strengthens the endpoint estimate $(b)$ of Theorem \ref{4unif rectif teo3b}. Moreover, in combination with the techniques used in \cite{MT2}, we think that the following theorem could be useful to derive $L^p$ ($1<p<\infty$) and endpoint estimates for $\VV_{\rho}\circ\TT^{\mu}$ and $\OO\circ\TT^{\mu}$ when $\mu$ is any $n$-dimensional AD regular uniformly $n$-rectifiable measure in $\R^d$, which would enhance \cite[Theorems 1.3 and 2.3]{MT2}. 

\begin{teo}\label{4unif rectif teo3}
Let $\rho>2$. Let $\Gamma\subset\R^d$ be an $n$-dimensional Lipschitz graph with slope strictly less than $1$ and set $\mu:=\HH^{n}_{\Gamma}$. Then, $\VV_{\rho}\circ\TT$ and
$\OO\circ\TT$ are bounded operators from $M(\R^d)$ to $L^{1,\infty}(\mu)$, i.e., there exist constants $C_1,C_2>0$ such that, for all $\lambda>0$ and all $\nu\in M(\R^d)$, 
$$\mu\{x\in\R^d:\,(\VV_{\rho}\circ\TT)\nu(x)>\lambda\}\leq\frac{C_1}{\lambda}\,\|\nu\|\quad\text{and}\quad
\mu\{x\in\R^d:\,(\OO\circ\TT)\nu(x)>\lambda\}\leq\frac{C_2}{\lambda}\,\|\nu\|.$$
Moreover, the constants $C_1$ and $C_2$ only depend on $n$, $d$, $K$, and the slope of $\Gamma$ (and on $\rho$ in the case of $C_1$). In particular, $C_2$ does not depend on the sequence that defines $\OO$. 
\end{teo}

\begin{remarko}
{\em We think that the assumption on the smallness of the slope of the Lipschitz graph in Theorems \ref{4unif rectif teo3b} and \ref{4unif rectif teo3} is just a technical obstruction due to the arguments we will employ in their proofs. As pointed out in the paragraph above Theorem \ref{4unif rectif teo3}, we expect that this assumption will be removed in the future.}
\end{remarko}

The following corollary is a direct consequence of Theorem \ref{4unif rectif teo3}.
\begin{coro}\label{coro pppp}
Let $E$ be an $\HH^n$ measurable $n$-rectifiable subset of $\R^d$ with $\HH^n(E)<\infty$, and let $K$ be an odd kernel satisfying $(\ref{4eq333})$. If $\nu\in M(\R^d)$, then the principal values $\lim_{\epsilon\searrow0}T_\epsilon\nu(x)$ exist for $\HH^n$ almost all $x\in E$.
\end{coro}

Given an $n$-rectifiable set $E\subset\R^d$ with $\HH^n(E)<\infty$, as far as the author knows, the existence $\HH^n_E$-a.e. of $\lim_{\epsilon\searrow0}T_\epsilon\nu(x)$ for $\nu\in M(\R^d)$ was already known for odd kernels $K\in\CC^\infty(\R^d\setminus\{0\})$ satisfying 
\begin{equation}\label{pppp}
|\nabla^j K(x)|\leq C_j|x|^{-n-j}
\end{equation} 
for {\em all} $j=0,1,2,3,\ldots$, or maybe assuming (\ref{pppp}) only for a finite but big number of $j$'s (see \cite[Theorems 20.15 and 20.27, Remarks 20.16 and 20.19]{Mattila-llibre} and the references therein). However, the result is new if one only asks (\ref{pppp}) for $j=0,1,2$, and so Corollary \ref{coro pppp} improves on previous results.

The plan of the paper is the following: In Section \ref{s_preliminaries} we state some preliminary results concerning a Calder\'on-Zygmund decomposition of general measures and about the Hausdorff measure of a Lipschitz graph on annuli. The proof of Theorem \ref{4unif rectif teo3} is given in Section \ref{4sec unif rectif teo3bis}, and in Section \ref{teorema Lp no suau s2} we prove Theorem \ref{4unif rectif teo3b}$(c)$. Finally, in Section \ref{s endsec} we complete the proof of Theorem \ref{4unif rectif teo3b} and we also prove Corollary \ref{coro pppp}.

\begin{remarko}
{\em We will only give the proof of Theorems \ref{4unif rectif teo3b} and \ref{4unif rectif teo3} for $\VV_\rho$, because the case of $\OO$ follows by very similar arguments and computations. The details are left for the reader.}
\end{remarko}

As usual, in the paper the letter `$C$' stands
for some constant which may change its value at different
occurrences, and which quite often only depends on $n$ and $d$. The notation $A\lesssim B$ ($A\gtrsim B$) means that
there is some constant $C$ such that $A\leq CB$ ($A\geq CB$),
with $C$ as above. Also, $A\approx B$ is equivalent to $A\lesssim B\lesssim A$.

\section{Preliminaries}\label {s_preliminaries}
\subsection{Calder\'{o}n-Zygmund decomposition for general measures}\label{4secdob}
Given a cube $Q$ in $\R^d$, we denote by $\ell(Q)$ the side length of $Q$. In this paper, the cubes are assumed to be closed and to have sides parallel to the coordinate axes.
Given $\nu\in M(\R^d)$, $ a >1$ and $ b > a ^n$, we say that a cube $Q$
is $( a , b )$-$|\nu|$-doubling if $|\nu|( a  Q)\leq b |\nu|(Q)$, where
$ a  Q$ is the cube concentric with $Q$ with side length $ a \ell(Q)$.
For definiteness, if $ a $ and $ b $ are not specified, by a doubling cube
we mean a $(2,2^{d+1})$-$|\nu|$-doubling cube. 

The following two lemmas are already known (see \cite{Tolsa10}, \cite{Tolsa9}, or \cite{Tolsa11} for example), but since they are essential in this paper, we give their proof for completeness. 

\begin{lema}\label{4lemcubdob}
Let $ b > a ^d$. If $\nu$ is a Radon measure in $\R^d$,  then for $\nu$-a.e. $x\in\R^d$ there
exists a sequence of $( a , b )$-$|\nu|$-doubling cubes $\{Q_k\}_k$ centered
at $x$ with $\ell(Q_k)\to0$ as $k\to\infty$.
\end{lema}

\begin{proof}[{\bf {\em Proof.}}]
Let $Z\subset\R^d$ be the set of points $x$ such that there
does not exist a sequence of $( a , b )$-$|\nu|$-doubling cubes $\{Q_k\}_{k\geq0}$ centered
at $x$ with side length decreasing to $0$; and let $Z_j\subset\R^d$ be the set of points $x$
such that there does not exist any $( a , b )$-$|\nu|$-doubling cube $Q$ centered
at $x$ with $\ell(Q)\leq 2^{-j}$. Clearly, $Z=\bigcup_{j\geq0}Z_j$. Thus, proving the lemma
is equivalent to showing that $\nu(Z_j)=0$ for every $j\geq0$.

Let $Q_0$ be a fixed cube with side length $2^{-j}$ and let $k\geq1$ be some integer.
For each $z\in Q_0\cap Z_j$, let $Q_z$ be a cube
centered at $z$ with side length $ a ^{-k}\ell(Q_0)$. Since the cubes $ a ^{h}Q_z$ are not
$( a , b )$-$|\nu|$-doubling for $h=0,\ldots,k-1$ and $ a ^kQ_z\subset 2Q_0$, we have
\begin{equation}\label{4eq112}
\nu(Q_z)\leq  b ^{-1}\nu( a  Q_z)\leq \cdots \leq  b ^{-k}\nu( a ^k Q_z)
\leq  b ^{-k}\nu(2Q_0).
\end{equation}

By Besicovitch's theorem, there exists a subfamily $\{z_m\}_m\subset Q_0\cap Z_j$ such that
$Q_0\cap Z_j\subset\bigcup_m Q_{z_m}$ and moreover  $\sum_{m}\chi_{Q_{z_m}}\leq P_d$. This is a finite family and the number $N$ of points $z_m$
can be easily bounded above as follows: if $\LL$ stands for the Lebesgue measure on $\R^d$,
$$N\,( a ^{-k}\ell(Q_0))^d =\sum_{m=1}^N \LL(Q_{z_m}) \leq P_d\LL(2Q_0) = P_d(2\ell(Q_0))^d.$$
Thus, $N\leq P_d2^d a ^{kd}.$
As a consequence, since $\{Q_{z_m}\}_{1\leq m\leq N}$ covers $Q_0\cap Z_j$,
by (\ref{4eq112}),
$$\nu(Q_0\cap Z_j)\leq \sum_{m=1}^N \nu(Q_z)
\leq  N  b ^{-k}\nu(2Q_0)\leq P_d2^d a ^{kd}
 b ^{-k}\nu(2Q_0).$$
Since $ b > a ^d$, the right hand side tends to $0$ as $k\to\infty$. Therefore
$\nu(Q_0\cap Z_j)=0$, and since the cube $Q_0$ is arbitrary, we are done.
\end{proof}

\begin{lema}[Calder\'{o}n-Zygmund decomposition]\label{4lema CZ}
Assume that $\mu:=\HH^n_{\Gamma\cap B}$, where $\Gamma$ is an $n$-dimensional Lipschitz graph and $B\subset\R^d$ is some fixed ball. For every $\nu\in M(\R^d)$ with compact support and every $\lambda>2^{d+1}\|\nu\|/\|\mu\|$, we have:
\begin{itemize}
\item[$(a)$] There exists a finite or countable collection of almost disjoint cubes $\{Q_j\}_j$ (that is, $\sum_j\chi_{Q_j}\leq C$) and a function $f\in L^1(\mu)$ such that
\begin{gather}
|\nu|(Q_j)>2^{-d-1}\lambda\mu(2Q_j),\label{4lema CZ 1}\\
|\nu|(\eta Q_j)\leq2^{-d-1}\lambda\mu(2\eta Q_j)\quad\text{for }\eta>2,\label{4lema CZ 2}\\
\nu=f\mu\text{ in }\R^d\setminus\Omega\text{ with }|f|\leq\lambda\;\,\mu\text{-a.e},\text{ where }\Omega={\textstyle\bigcup_j}Q_j.\label{4lema CZ 3}
\end{gather}
\item[$(b)$]For each $j$, let $R_j:=6Q_j$ and denote $w_j:=\chi_{Q_j}\big(\sum_k\chi_{Q_k}\big)^{-1}$. Then, there exists a family of functions $\{b_j\}_j$ with $\supp b_j\subset R_j$ and with constant sign satisfying
\begin{gather}
\int b_j\,d\mu=\int w_j\,d\nu,\label{4lema CZ 4}\\
\| b_j\|_{L^\infty(\mu)}\mu(R_j)\leq C|\nu|(Q_j),\text{ and}\label{4lema CZ 5}\\
{\textstyle\sum_j}|b_j|\leq C_0\lambda\quad
\text{(where $C_0$ is some absolute constant).}\label{4lema CZ 6}
\end{gather}
\end{itemize}
\end{lema}

\begin{proof}[{\bf{\em Proof of} Lemma \ref{4lema CZ}$(a)$}]
Let $H$ be the set of those points from $\supp\mu\cup\supp\nu$ such that there exists
some cube $Q$ centered at $x$ satisfying
$|\nu|(Q)>2^{-d-1}\lambda\mu(2Q).$
For each $x\in H$, let $Q_x$ be a cube centered at $x$ such that the preceding inequality holds
for $Q_x$ but fails for the cubes $Q$ centered at $x$ with $\ell(Q)>2\ell(Q_x)$.
Notice that the condition $\lambda>2^{d+1}\,\|\nu\|/\|\mu\|$ guaranties the existence of $Q_x$.

Since $H$ is bounded (because $\mu$ and $\nu$ are compactly supported), we can apply Besicovitch's covering theorem
to get a finite or countable almost disjoint subfamily of cubes
$\{Q_j\}_j\subset\{Q_x\}_{x\in H}$ which cover $H$ and satisfy (\ref{4lema CZ 1}) and (\ref{4lema CZ 2})
by construction.

To prove (\ref{4lema CZ 3}), denote by
$Z$ be the set of points $y\in\supp\nu$ such there does not exist a sequence of
$(2,2^{d+1})$-$|\nu|$-doubling cubes centered at $y$ with side length tending to $0$, so that
$|\nu|(Z)=0$, by Lemma \ref{4lemcubdob}. By the definitions of $H$ and $Z$,
for every $x\in \supp\nu\setminus (H\cup Z)$, there exists a sequence of
$(2,2^{d+1})$-$|\nu|$-doubling cubes $P_k$ centered at $x$, with $\ell(P_k)\to0$, such that
$|\nu|(P_k)\leq2^{-d-1}\lambda\mu(2P_k),$ and thus
$|\nu|(2P_k)\leq 2^{d+1}|\nu|(P_k)\leq\lambda\mu(2P_k).$
This implies that $\chi_{\R^d\setminus(H\cup Z)}\nu$ is absolutely continuous with respect to
$\mu$ and that $\chi_{\R^d\setminus H}\nu=\chi_{\R^d\setminus(H\cup Z)}\nu=f\mu$ with $|f|\leq\lambda$ $\mu$-a.e., by the Lebesgue-Radon-Nikodym theorem (see \cite[pages 36 to 39]{Mattila-llibre}, for instance).
\end{proof}

\begin{proof}[{\bf{\em Proof of} Lemma \ref{4lema CZ}$(b)$}] 
Assume first that the family of cubes $\{Q_j\}_j$ is finite.
Then we may suppose that this family of cubes is ordered in such a way
that the sizes of the cubes $R_j$ are non decreasing (i.e. $\ell(R_{j+1})\geq\ell(R_j)$).
The functions $b_j$ that we will construct will be of the form $b_j=c_j\,\chi_{A_j}$, with $c_j\in\R$ and $A_j\subset R_j$. We set $A_1=R_1$ and
$b_1 := c_1\,\chi_{R_1},$
where the constant $c_1$ is chosen so that $\int_{Q_1}w_1\,d\nu=\int b_1\,d\mu$.

Suppose that $b_1,\ldots,b_{k-1}$ have been constructed,
satisfy (\ref{4lema CZ 4}) and $\sum_{j=1}^{k-1} |b_j|\leq
C_0\,\lambda,$  where $C_0$ is some constant which will be fixed below.
Let $R_{s_1},\ldots,R_{s_m}$ be the subfamily of
$R_1,\ldots,R_{k-1}$ such that $R_{s_i}\cap R_k \neq \emptyset$.
As $\ell(R_{s_i}) \leq \ell(R_k)$ (because of the non decreasing sizes of $R_j$),
we have $R_{s_i} \subset 3R_k$. Taking into account that
$\int |b_j|\,d\mu \leq |\nu|(Q_j)$ for $j=1,\ldots,k-1$
by (\ref{4lema CZ 4}), and using (\ref{4lema CZ 2}) and that $\mu(6R_k)\leq C\mu(R_k)$ (because $\frac{1}{2}R_k=3Q_k$ intersects $\supp\mu$ by (\ref{4lema CZ 2})), we get
\begin{align*}
\sum_i\int|b_{s_i}|\,d\mu\leq\sum_i|\nu|(Q_{s_i})
\leq C|\nu|(3R_k)\leq C\lambda\mu(6R_k)\leq C_2\lambda\mu(R_k).
\end{align*}
Therefore,
$\mu\left\{x\in R_k\,:\,{\textstyle \sum_i}|b_{s_i}(x)|>2C_2\lambda\right\}\leq
\mu(R_k)/2.$ So, if we set
$$A_k:=\left\{x\in R_k\,:\,{\textstyle\sum_i}|b_{s_i}(x)|\leq
2C_2\lambda\right\},$$ then $\mu(A_k)\geq\mu(R_k)/2.$

The constant $c_k$ is chosen so that for $b_k=c_k\chi_{A_k}$
we have $\int b_k\,d\mu=\int_{Q_k}w_k\,d\nu$.  Then we obtain, by (\ref{4lema CZ 2}),
$$|c_k|\leq\frac{|\nu|(Q_k)}{\mu(A_k)}
\leq\frac{2|\nu|(\frac{1}{2}R_k)}{\mu(R_k)}\leq C_3\lambda$$
(this calculation also applies to $k=1$). Thus,
$|b_k|+\sum_{i} |b_{s_i}| \leq (2C_2+C_3)\,\lambda.$
If we choose $C_0=2C_2+C_3$, (\ref{4lema CZ 6}) follows.

Now it is easy to check that (\ref{4lema CZ 5}) also holds. Indeed we have
$$\|b_j\|_{L^\infty(\mu)}\mu(R_j)\leq C|c_j|\mu(A_j)
=C\bigg|\int_{Q_j}w_j\,d\nu\bigg|\leq C|\nu|(Q_j).$$

Suppose now that the collection of cubes $\{Q_j\}_j$ is not finite.
For each fixed $N$ we consider the family of cubes $\{Q_j\}_{1\leq j \leq N}$.
Then, as above, we construct functions $b_1^N,\ldots,b_N^N$ with
$\supp(b_j^N)\subset R_j$ satisfying
$\int b_j^N\,d\mu=\int_{Q_j}w_j\,d\nu,$
$\sum_{j=1}^N|b_j^N|\leq C_0\,\lambda$
and
$\|b_j^N\|_{L^\infty(\mu)}\mu(R_j)\leq C|\nu|(Q_j).$
Notice that the sign of $b_j^N$ equals the sign of $\int w_j\,d\nu$ and
so it does not depend on $N$.

Then there is a subsequence
$\{b_1^k\}_{k\in I_1}$ which is convergent in the weak $\ast$ topology of
$L^\infty(\mu)$ to some function
$b_1\in L^\infty(\mu)$. Now we can consider a subsequence
$\{b_2^k\}_{k\in I_2}$ with $I_2\subset I_1$ which
is also convergent in the weak $\ast$ topology of $L^\infty(\mu)$ to some
function $b_2\in L^\infty(\mu)$.
In general, for each $j$ we consider a subsequence
$\{b_j^k\}_{k\in I_j}$ with $I_j\subset I_{j-1}$ that converges
in the weak $\ast$ topology of $L^\infty(\mu)$ to some function
$b_j\in L^\infty(\mu)$. It is easily checked that the functions
$b_j$ satisfy the required properties.
\end{proof}

\subsection{Hausdorff measure of Lipschitz graphs on annuli}\label{ss pendent petita}
Given $z\in\R^d$ and $0<a\leq b$, let $A(z,a,b)\subset\R^d$ denote the closed annulus centered at $z$ and with inner radius $a$ and outer radius $b$.
This subsection is devoted to the proof of the following lemma, which yields a key estimate to derive Theorems \ref{4unif rectif teo3b} and \ref{4unif rectif teo3}.
\begin{lema}\label{lema pendent petita3}
Let $\Gamma:=\{x\in \R^d\,:\,x=(y,\A(y)),\, y\in\R^n\}$ be the graph of a Lipschitz function $\A:\R^n\to\R^{d-n}$ such that $\Lip(\A)<1$. Then, there exists $C>0$ depending on $n$, $d$, and $\Lip(\A)$, such that $\HH^n_\Gamma(A(z,a,b))\leq C(b-a)b^{n-1}$ for all $z\in\Gamma$ and all $0<a\leq b$.
\end{lema}

We need the following auxiliary result.
\begin{lema} \label{lema pendent petita1}
Let $1\leq n<d$. For $x:=(x_1,\ldots,x_d)\in\R^d$ we denote $$x_H:=(x_1,\ldots,x_n,0,\ldots,0)\in\R^d\quad\text{and}\quad
x_V:=(0,\ldots,0,x_{n+1},\ldots,x_d)\in\R^d.$$
Given $x,y\in\R^d\setminus\{0\}$, if there exists $0<s<1$ such that $|x_V|\leq s|x_H|$, $|y_V|\leq s|y_H|$, and $|x_V-y_V|\leq s|x_H-y_H|$, then there exists $C>0$ depending only on $s$ such that
\begin{equation}\label{pendent cota}
|x_V-y_V|\leq C\bigg|\frac{|x|}{|x_H|}\,x_H-\frac{|y|}{|y_H|}\,y_H\bigg|.
\end{equation}
\end{lema}
\begin{proof}[{\bf {\em Proof.}}]
We set $\Phi(x,y):=\big||x||x_H|^{-1}x_H-|y||y_H|^{-1}y_H\big|$. Since $\Phi$ is symmetric in $x$ and $y$, we can assume that $|x_H|\leq|y_H|$. If $\langle\cdot,\cdot\rangle$ denotes the scalar product in $\R^d$, using the polarization identity,
\begin{equation*}
\begin{split}
\Phi(x,&y)^2=|x|^2+|y|^2-2|x||x_H|^{-1}|y||y_H|^{-1}\langle x_H,y_H\rangle\\
&=|x|^2+|y|^2+|x||x_H|^{-1}|y||y_H|^{-1}\big(|x_H-y_H|^2-|x_H|^2-|y_H|^2\big)\\
&=|x|^2+|y|^2-2|x||y|+|x||x_H|^{-1}|y||y_H|^{-1}
\big(|x_H-y_H|^2-|x_H|^2-|y_H|^2+2|x_H||y_H|\big)\\
&=\big(|x|-|y|\big)^2+|x||x_H|^{-1}|y||y_H|^{-1}
\big(|x_H-y_H|^2-(|x_H|-|y_H|)^2\big).
\end{split}
\end{equation*}
Since $|x_H-y_H|^2-(|x_H|-|y_H|)^2\geq0$, $|x_H|\leq|x|$, and $|y_H|\leq|y|$, we have
\begin{equation}\label{pendent1}
\begin{split}
\Phi(x,y)^2\geq\big(|x|-|y|\big)^2+|x_H-y_H|^2-(|x_H|-|y_H|)^2.
\end{split}
\end{equation}

Assume that $2|x|\leq|y|$. Then, using (\ref{pendent1}),
\begin{equation*}
\begin{split}
|x_V-y_V|\leq|x|+|y|\leq\frac{3}{2}|y|=3\Big(|y|-\frac{1}{2}\,|y|\Big)
\leq3(|y|-|x|)\leq3\Phi(x,y),
\end{split}
\end{equation*}
and we obtain (\ref{pendent cota}). By the same arguments, if $2|y|\leq|x|$, then
$|x_V-y_V|\leq3\Phi(x,y)$ and (\ref{pendent cota}) holds. Thus, from now on we assume $\frac{1}{2}|x|\leq|y|\leq2|x|$.

Let $0<\delta<1$ be a small number that will be fixed below. Assume that
$(1-\delta)|x_H-y_H|\geq\big||y_H|-|x_H|\big|$. Then, by (\ref{pendent1}),
\begin{equation*}
\begin{split}
\Phi(x,y)^2&\geq|x_H-y_H|^2-(|x_H|-|y_H|)^2
\geq|x_H-y_H|^2-(1-\delta)^2|x_H-y_H|^2\\
&=\delta(2-\delta)|x_H-y_H|^2\geq\delta(2-\delta) s^{-2}|x_V-y_V|^2,
\end{split}
\end{equation*}
and then (\ref{pendent cota}) holds with $C=s/\sqrt{\delta(2-\delta)}$.

Therefore, we can suppose that $(1-\delta)|x_H-y_H|\leq\big||y_H|-|x_H|\big|=|y_H|-|x_H|$, since we are also assuming $|x_H|\leq|y_H|$.
If we set $z=y-x$, we have $(1-\delta)|z_H|\leq|x_H+z_H|-|x_H|$, so
$(1-\delta)|z_H|+|x_H|\leq|x_H+z_H|$. Hence,
\begin{equation*}
\begin{split}
(1-\delta)^2|z_H|^2+|x_H|^2+2(1-\delta)|z_H||x_H|
&=\big((1-\delta)|z_H|+|x_H|\big)^2\\
&\leq|x_H+z_H|^2
=|x_H|^2+|z_H|^2+2\langle x_H,z_H\rangle
\end{split}
\end{equation*}
and we obtain
\begin{equation}\label{pendent2}
\begin{split}
\langle x_H,z_H\rangle\geq-\frac{1}{2}\,\delta(2-\delta)|z_H|^2+(1-\delta)|z_H||x_H|.
\end{split}
\end{equation}

Using (\ref{pendent2}), that $\langle x_V,z_V\rangle\geq-|x_V||z_V|$, and that $|x_V|\leq s|x_H|$ and $|z_V|\leq s|z_H|$, we get
\begin{equation}\label{pendent3}
\begin{split}
\langle x, z\rangle&=\langle x_H+x_V, z_H+z_V\rangle=\langle x_H, z_H\rangle+\langle x_V,z_V\rangle\\
&\geq-\frac{1}{2}\,\delta(2-\delta)|z_H|^2+(1-\delta)|z_H||x_H|-|x_V||z_V|\\
&\geq-\frac{1}{2}\,\delta(2-\delta)|z_H|^2+(1-\delta-s^2)|z_H||x_H|.
\end{split}
\end{equation}
Notice that, if $\delta>0$ is small enough depending on $s$, then
$-\frac{1}{4}\,(1-s^2)(1+s^2)^{-1}<-\frac{3}{2}\,\delta(2-\delta)<0$ and
$1-\delta-s^2>\frac{1}{2}\,(1-s^2)$. Let $\gamma(x,z)$ be the angle
between $x$ and $z$ (by definition, $0\leq\gamma(x,z)\leq\pi$). Using that $\langle x,
z\rangle=|x||z|\cos(\gamma(x,z))$, that $|x|\leq\sqrt{1+s^2}|x_H|$ and $|z|\leq\sqrt{1+s^2}|z_H|$, and that
$|z|\leq|x|+|y|\leq3|x|$, we finally obtain from (\ref{pendent3})
that
\begin{equation*}
\begin{split}
\cos(\gamma(x,z))&\geq-\frac{1}{2}\,\delta(2-\delta)|z_H|^2|x|^{-1}|z|^{-1}+(1-\delta-s^2)|z_H||x_H||x|^{-1}|z|^{-1}\\
&\geq-\frac{3}{2}\,\delta(2-\delta)+(1-\delta-s^2)(1+s^2)^{-1}\geq\frac{1}{4}\,(1-s^2)(1+s^2)^{-1}=:a.
\end{split}
\end{equation*}

Notice that $a>0$, because $0<s<1$ by hypothesis. Hence, since $\cos(\gamma(-x,y-x))=\cos(\gamma(-x,z))=-\cos(\gamma(x,z))$ (because $z=y-x$ and $\langle-x,z\rangle=-\langle x,z\rangle$), we have $c_0:=\cos(\gamma(-x,y-x))\leq -a<0$ (notice that $c_0\leq0$ implies that $|x|\leq|y|$). By the cosines theorem, $|y|^2=|x|^2-|y-x|^2-2|x||y-x|c_0$. Since $c_0<0$, we solve the second degree equation in $|y-x|$ and we obtain
\begin{equation*}
\begin{split}
|y-x|&=\sqrt{|y|^2-|x|^2(1-c_0^2)}-|x||c_0|
=\frac{|y|^2-|x|^2(1-c_0^2)-|x|^2c_0^2}{\sqrt{|y|^2-|x|^2(1-c_0^2)}+|x||c_0|}\\
&=\frac{(|y|-|x|)(|y|+|x|)}{\sqrt{|y|^2-|x|^2(1-c_0^2)}+|x||c_0|}
\leq\frac{(|y|-|x|)(|y|+|x|)}{|x||c_0|}\leq(|y|-|x|)\frac{3}{a},
\end{split}
\end{equation*}
where we also used that $|y|\leq2|x|$ in the last inequality. Therefore, by (\ref{pendent1}),
$$|x_V-y_V|\leq|x-y|\leq\frac{3}{a}\,(|y|-|x|)\leq\frac{3}{a}\,\Phi(x,y),$$
and (\ref{pendent cota}) follows with $C=3/a$, where $a>0$ only depends on $s$. This completes the proof of the lemma.
\end{proof}

\begin{proof}[{\bf {\em Proof} of Lemma \ref{lema pendent petita3}}]
We keep the notation introduced in Lemma \ref{lema pendent petita1}. Fix $z\in\Gamma$. We can assume that $z=0$, by taking a translation of $\Gamma$ if it is necessary.

For $x\in\R^d$ with $x_H\neq0$, consider the map 
$$\Upsilon(x):=\frac{|x|}{|x_H|}\,x_H+x_V
=\sqrt{1+\frac{|x_V|^2}{|x_H|^2}}\,x_H+x_V.$$ 
It is not difficult to show that $\Upsilon$ is a bilipschitz mapping from (a neighborhood of) the cone $$L:=\{x\in\R^d\setminus\{0\}\,:\,|x_V|\leq\Lip(\A)|x_H|\}$$ to (a neighborhood of) the cone $$L':=\{x\in\R^d\setminus\{0\}\,:\,|x_V|\leq\Lip(\A)(1+\Lip(\A)^2)^{-1/2}|x_H|\},$$ whose inverse equals $$\Upsilon^{-1}(x)=\sqrt{1-\frac{|x_V|^2}{|x_H|^2}}\,x_H+x_V.$$  
Moreover, when $\Upsilon$ and $\Upsilon^{-1}$ are restricted to $L$ and $L'$ respectively, $\Lip(\Upsilon)$ and $\Lip(\Upsilon^{-1})$ only depend on $n$, $d$, and $\Lip(\A)$. Hence, since $\Gamma\subset L\cup\{0\}$, for any $0<a\leq b$ we have $$\HH^n_\Gamma(A(0,a,b))=\HH^n(\Gamma\cap A(0,a,b))\approx\HH^n(\Upsilon(\Gamma\cap A(0,a,b))).$$

Consider the set $\Upsilon(\Gamma)$. Since $\Gamma$ has slope smaller than $1$ (i.e. $\Lip(\A)<1$), by Lemma \ref{lema pendent petita1} there exists a constant $C>0$ depending only on $n$, $d$, and $\Lip(\A)$ such that for any two points $x,y\in\Upsilon(\Gamma)$ one has $|x_V-y_V|\leq C|x_H-y_H|$. Then, it is known that $\Upsilon(\Gamma)$ is contained in the $n$-dimensional graph $\Gamma'$ of some Lipschitz function (see for example the proof of \cite[Lemma 15.13]{Mattila-llibre}). Notice also that, given $0<a\leq b$, $\Upsilon(L\cap A(0,a,b))\subset\{x\in\R^d:\,a\leq|x_H|\leq b\}$. Therefore,
\begin{equation*}
\begin{split}
\HH^n_\Gamma(A(0,a,b))&\approx\HH^n(\Upsilon(\Gamma\cap A(0,a,b)))
\leq\HH^n(\Gamma'\cap\{x\in\R^d:\,a\leq|x_H|\leq b\})\lesssim(b-a)b^{n-1},
\end{split}
\end{equation*}
and the lemma is proved.
\end{proof}

\begin{remarko}{\em 
With a little more of effort, one can show that $\Upsilon(\Gamma)$ is actually a Lipschitz graph. We omit the details.}
\end{remarko}

\begin{remarko}\label{rem499}{\em
Lemma \ref{lema pendent petita3} is sharp in the sense that the estimate fails if $\Lip(\A)\geq1$ (notice that the constant $C$ in Lemma \ref{lema pendent petita1} for $s=\Lip(\A)$ is bigger than $(1+\Lip(\A)^2)/(1-\Lip(\A)^2)$). Given $\epsilon>0$, one can easily construct a Lipschitz graph $\Gamma$ such that $1<\Lip(\A)<1+\epsilon$ and such that, for some $z\in\Gamma$ and $r>0$, $\Gamma$ contains a set $P\subset\partial B(z,r)$ with $\HH^n_\Gamma(P)>0$. Then, if Lemma \ref{lema pendent petita3} were true for $\Gamma$, we would have
$0<\HH^n_\Gamma(P)\leq\HH^n_\Gamma(A(z,r-\delta,r+\delta))\lesssim2\delta(r+\delta)^{n-1}$, and we would have a contradiction by letting $\delta\to0$.
By a similar argument, one can also show that the lemma fails in the limiting case $\Lip(\A)=1$.}
\end{remarko}

\section{$\VV_\rho\circ\TT$ is a bounded operator from $M(\R^d)$ to $L^{1,\infty}(\HH^n_\Gamma)$}\label{4sec unif rectif teo3bis}

This section is devoted to the proof of Theorem \ref{4unif rectif teo3}, which is based on a nontrivial modification of the proof of \cite[Theorem B]{CJRW-singular integrals} using the Calder\'{o}n-Zygmund decomposition developed in Subsection \ref{4secdob}. 

\begin{proof}[{\bf {\em Proof of} Theorem \ref{4unif rectif teo3}}]
Set $\mu:=\HH^n_{\Gamma\cap B}$, where $B$ is some fixed ball in $\R^d$. Let $\nu\in M(\R^d)$ be a finite complex Radon measure with compact support and $\lambda>2^{d+1}\|\nu\|/\|\mu\|$. We will show that
\begin{equation}\label{4cota mesuresL1debil}
\mu\big(\big\{x\in\R^d\,:\,(\VV_\rho\circ\TT)\nu(x)>\lambda\big\}\big)\leq\frac{C}{\lambda}\,\|\nu\|,
\end{equation}
where $C>0$ depends on $n$, $d$, $K$, $\rho$ and $\Gamma$, but not on $B\subset\R^d$. Let us check that (\ref{4cota mesuresL1debil}) implies that $\VV_\rho\circ\TT$ is bounded from $M(\R^d)$ into $L^{1,\infty}(\HH^n_\Gamma)$. 
Suppose that $\nu$ is not compactly supported. Set $\nu_N = \chi_{B(0,N)}\,\nu$.
Let $N_0$ be such that $\supp\mu\subset B(0,N_0)$. Then it is not hard to show that, for $x\in\supp\mu$,
$$|(\VV_\rho\circ\TT)\nu(x)-(\VV_\rho\circ\TT)\nu_N(x)|\leq C\,\frac{|\nu|(\R^d\setminus B(0,N))}{N-N_0},$$
thus $(\VV_\rho\circ\TT)\nu_N(x)\to (\VV_\rho\circ\TT)\nu(x)$ for all $x\in \supp\mu$ uniformly, and since the estimate (\ref{4cota mesuresL1debil}) 
 holds by assumption
for $\nu_N$, letting $N\to\infty$, we deduce that it also holds for $\nu$. Now, by increasing the size of the ball $B$ and monotone convergence, (\ref{4cota mesuresL1debil}) yields 
$$\HH^n_\Gamma\big(\big\{x\in\R^d\,:\,(\VV_\rho\circ\TT)\nu(x)>\lambda\big\}\big)\leq \frac{C}{\lambda}\,\|\nu\|,$$ as desired. Thus, we only have to verify (\ref{4cota mesuresL1debil}) for all compactly supported $\nu$.

Let $\{Q_j\}_j$ be the almost disjoint
family of cubes of Lemma \ref{4lema CZ}, and set $\Omega:=\bigcup_jQ_j$ and $R_j:=6Q_j$.
Then we can write $\nu=g\mu+\nu_b$, with
$$g\mu=\chi_{\R^d\setminus\Omega}\nu+ \sum_j b_j\mu\quad\text{and}\quad
\nu_b=\sum_j\nu_b^j:=\sum_j\left(w_j\nu-b_j\mu\right),$$
where the functions $b_j$ satisfy (\ref{4lema CZ 4}), (\ref{4lema CZ 5}), and (\ref{4lema CZ 6}), and
$w_j=\chi_{Q_j}\big(\sum_k \chi_{Q_k}\big)^{-1}$.

By the subadditivity of $\VV_\rho\circ\TT$, we have
\begin{equation}\label{4cota mesuresL1debil 2}
\begin{split}
\mu\big(\big\{x&\in\R^d\,:\,(\VV_\rho\circ\TT)\nu(x)>\lambda\big\}\big)\\
&\leq\mu\big(\big\{x\in\R^d\,:\,(\VV_\rho\circ\TT^\mu)g(x)>\lambda/2\big\}\big)
+\mu\big(\big\{x\in\R^d\,:\,(\VV_\rho\circ\TT)\nu_b(x)>\lambda/2\big\}\big).
\end{split}
\end{equation}

Since $\VV_\rho\circ\TT^{\HH^n_\Gamma}$ is bounded in $L^2(\HH^n_\Gamma)$ by Theorem \ref{4main theorem lip}, it is easy to show that $\VV_\rho\circ\TT^{\mu}$ is bounded in $L^2(\mu)$, with a bound independent of $B$. Notice that $|g|\leq C\lambda$ by (\ref{4lema CZ 3}) and (\ref{4lema CZ 6}). Then, using (\ref{4lema CZ 5}),
\begin{equation}\label{4cota mesuresL1debil 3}
\begin{split}
\mu\big(\big\{x\in\R^d\,:\,(\VV_\rho\circ\TT^\mu)g(x)>\lambda/2\big\}\big)
&\lesssim\frac{1}{\lambda^{2}}\int|(\VV_\rho\circ\TT^\mu)g|^2\,d\mu
\lesssim\frac{1}{\lambda^{2}}\int|g|^2\,d\mu\\
&\lesssim\frac{1}{\lambda}\int|g|\,d\mu
\leq\frac{1}{\lambda}\bigg(|\nu|(\R^d\setminus\Omega)+\sum_j\int_{R_j}|b_j|\,d\mu\bigg)\\
&\leq\frac{1}{\lambda}\bigg(|\nu|(\R^d\setminus\Omega)+\sum_j|\nu|(Q_j)\bigg)\leq\frac{C}{\lambda}\,\|\nu\|.
\end{split}
\end{equation}

Set $\wih\Omega:=\bigcup_j2Q_j$. By (\ref{4lema CZ 1}), we have
$\mu(\wih\Omega)\leq\sum_j\mu(2Q_j)\lesssim\lambda^{-1}\sum_j|\nu|(Q_j)\lesssim\lambda^{-1}\|\nu\|$. We are going to show that
\begin{equation}\label{4cota mesuresL1debil 1}
\mu\big(\big\{x\in\R^d\setminus\wih\Omega\,:\,(\VV_\rho\circ\TT)\nu_b(x)>\lambda/2\big\}\big)\leq\frac{C}{\lambda}\,\|\nu\|,
\end{equation}
and then (\ref{4cota mesuresL1debil}) is a direct consequence of (\ref{4cota mesuresL1debil 2}), (\ref{4cota mesuresL1debil 3}), (\ref{4cota mesuresL1debil 1}) and the estimate $\mu(\wih\Omega)\lesssim\lambda^{-1}\|\nu\|$.

For simplicity of notation, given $0<\epsilon\leq\delta$ and $t\in\R^d$, we set $\chi_\epsilon^\delta(t):=\chi_{(\epsilon,\delta]}(|t|),$ so 
$$T_\epsilon\nu_b(x)-T_\delta\nu_b(x)=\int\chi_\epsilon^\delta(x-y)K(x-y)\,d\nu_b(y)=(K\chi_\epsilon^\delta*\nu_b)(x).$$
Given $x\in\supp\mu$, let $\{\epsilon_m\}_{m\in\Z}$ be a decreasing sequence of positive numbers (which depends on $x$, i.e. $\epsilon_m\equiv\epsilon_m(x)$) such that
\begin{equation}\label{4cota mesuresL1debil 4b}
(\VV_\rho\circ\TT)\nu_b(x)
\leq2\bigg(\sum_{m\in\Z}|(K\chi_{\epsilon_{m+1}}^{\epsilon_m}*\nu_b)(x)|^\rho\bigg)^{1/\rho}.
\end{equation}
If $R_j\cap A(x,\epsilon_{m+1},\epsilon_m)=\emptyset$ then $(K\chi_{\epsilon_{m+1}}^{\epsilon_m}*\nu_b^j)(x)=0$, so by (\ref{4cota mesuresL1debil 4b}) and the triangle inequality,
\begin{equation*}
\begin{split}
(\VV_\rho\circ\TT)\nu_b(x)
&\leq2\bigg(\sum_{m\in\Z}\bigg|\sum_{j:\,R_j\subset A(x,\epsilon_{m+1},\epsilon_m)}
(K\chi_{\epsilon_{m+1}}^{\epsilon_m}*\nu_b^j)(x)\bigg|^\rho\bigg)^{1/\rho}\\
&\quad+2\bigg(\sum_{m\in\Z}\bigg|\sum_{j:\,R_j\cap\partial A(x,\epsilon_{m+1},\epsilon_m)\neq\emptyset}
(K\chi_{\epsilon_{m+1}}^{\epsilon_m}*\nu_b^j)(x)\bigg|^\rho\bigg)^{1/\rho}\\
&=:2\big(IS(x)+BS(x)\big),
\end{split}
\end{equation*}
and then,
\begin{equation}\label{4cota mesuresL1debil 5b}
\begin{split}
\mu\big(\big\{x\in\R^d&\setminus\wih\Omega\,:\,(\VV_\rho\circ\TT)\nu_b(x)>\lambda/2\big\}\big)\\
&\leq\mu\big(\big\{x\in\R^d\setminus\wih\Omega\,:\,IS(x)>\lambda/8\big\}\big)
+\mu\big(\big\{x\in\R^d\setminus\wih\Omega\,:\,BS(x)>\lambda/8\big\}\big).
\end{split}
\end{equation}

Let us estimate first $\mu\big(\big\{x\in\R^d\setminus\wih\Omega\,:\,IS(x)>\lambda/8\big\}\big)$.
Since the $\ell^\rho$-norm is not bigger than the $\ell^1$-norm for $\rho\geq1$,
\begin{equation}\label{4cota mesuresL1debil 9b}
\begin{split}
IS(x)&\leq\sum_{m\in\Z}\bigg|\sum_{j:\,R_j\subset A(x,\epsilon_{m+1},\epsilon_m)}
(K\chi_{\epsilon_{m+1}}^{\epsilon_m}*\nu_b^j)(x)\bigg|\\
&\leq\sum_{m\in\Z}\,\sum_{j:\,R_j\subset A(x,\epsilon_{m+1},\epsilon_m)}
\bigg|\int\chi_{\epsilon_{m+1}}^{\epsilon_m}(x-y)K(x-y)\,d\nu_b^j(y)\bigg|\\
&=\sum_{j}\sum_{m\in\Z:\,A(x,\epsilon_{m+1},\epsilon_m)\supset R_j}
\bigg|\int\chi_{\epsilon_{m+1}}^{\epsilon_m}(x-y)K(x-y)\,d\nu_b^j(y)\bigg|\\
&\leq\sum_{j}\chi_{\R^d\setminus R_j}(x)\bigg|\int K(x-y)\,d\nu_b^j(y)\bigg|.
\end{split}
\end{equation}
Notice that
\begin{equation}\label{4cota mesuresL1debil 8b}
\begin{split}
\int_{\R^d\setminus\wih\Omega}&\chi_{\R^d\setminus R_j}(x)\bigg|\int K(x-y)\,d\nu_b^j(y)\bigg|\,d\mu(x)
\leq\int_{\R^d\setminus R_j}\bigg|\int K(x-y)\,d\nu_b^j(y)\bigg|\,d\mu(x)\\
&\leq\int_{\R^d\setminus2R_j}\bigg|\int K(x-y)\,d\nu_b^j(y)\bigg|\,d\mu(x)
+\int_{2R_j\setminus R_j}\bigg|\int K(x-y)\,d\nu_b^j(y)\bigg|\,d\mu(x).
\end{split}
\end{equation}
On one hand, by (\ref{4lema CZ 5}) and using the $L^2(\mu)$ boundedness of the maximal operator $T_*^\mu$ (recall that $\mu=\HH^n_{\Gamma\cap B}$, where $\Gamma$ is a Lipschitz graph and $B$ is a ball) and that $\mu(2R_j)\leq C\mu(R_j)$ (because $\frac{1}{2}R_j\cap\supp\mu\neq\emptyset$), we get
\begin{equation}\label{4cota mesuresL1debil 6b}
\begin{split}
\int_{2R_j\setminus R_j}\bigg|\int K(x-y)b_j(y)\,d\mu(y)\bigg|\,d\mu(x)
&\leq\int_{2R_j\setminus R_j}T_*^\mu b_j\,d\mu\\
&\leq\bigg(\int_{2R_j}(T_*^\mu b_j)^2\,d\mu\bigg)^{1/2}\mu(2R_j)^{1/2}
\\&\lesssim\|b_j\|_{L^2(\mu)}\mu(2R_j)^{1/2}
\lesssim\|b_j\|_{L^\infty(\mu)}\mu(R_j)\\
&\lesssim|\nu|(Q_j).
\end{split}
\end{equation}
On the other hand, since $\supp w_j\subset Q_j=\frac{1}{6}R_j$ and $|w_j|\leq1$, if $x\in2R_j\setminus R_j$ we have
$\int|K(x-y)w_j(y)|\,d|\nu|(y)\lesssim|\nu|(Q_j)|x-z_j|^{-n}$, where $z_j$ denotes the center of $R_j$. Hence, using again that $\mu(2R_j)\leq C\mu(R_j)\leq C\ell(R_j)^n$,
\begin{equation}\label{4cota mesuresL1debil 7b}
\begin{split}
\int_{2R_j\setminus R_j}\bigg|\int K(x-y)w_j(y)\,d\nu(y)\bigg|\,d\mu(x)
&\leq\int_{2R_j\setminus R_j}\int |K(x-y)w_j(y)|\,d|\nu|(y)\,d\mu(x)\\
&\lesssim|\nu|(Q_j)\int_{2R_j\setminus R_j}|x-z_j|^{-n}\,d\mu(x)\\
&\lesssim|\nu|(Q_j)\ell(R_j)^{-n}\mu(2R_j)\lesssim|\nu|(Q_j).
\end{split}
\end{equation}

Since $\nu_b^j(R_j)=0$, $\supp\nu_b^j\subset R_j$, and $\|\nu_b^j\|\lesssim|\nu|(Q_j)$ by (\ref{4lema CZ 5}), we have
\begin{equation*}
\begin{split}
\int_{\R^d\setminus 2R_j}\bigg|\int K(x-y)\,d\nu_b^j(y)\bigg|\,d\mu(x)
&\leq\int_{\R^d\setminus 2R_j}\int_{R_j}|K(x-y)-K(x-z_j)|\,d|\nu_b^j|(y)\,d\mu(x)\\
&\lesssim\int_{\R^d\setminus 2R_j}\int_{R_j}\frac{|y-z_j|}{|x-z_j|^{n+1}}\,d|\nu_b^j|(y)\,d\mu(x)\\
&\lesssim\|\nu_b^j\|\int_{\R^d\setminus 2R_j}\frac{\ell(R_j)}{|x-z_j|^{n+1}}\,d\mu(x)\lesssim\|\nu_b^j\|
\lesssim|\nu|(Q_j).
\end{split}
\end{equation*}

Combining this last estimate with (\ref{4cota mesuresL1debil 6b}), (\ref{4cota mesuresL1debil 7b}), and the fact that $\nu_b^j=w_j\nu-b_j\mu$, from (\ref{4cota mesuresL1debil 8b}) we obtain that
$$\int_{\R^d\setminus\wih\Omega}\chi_{\R^d\setminus R_j}(x)\bigg|\int K(x-y)\,d\nu_b^j(y)\bigg|\,d\mu(x)
\lesssim|\nu|(Q_j).$$
Finally, using (\ref{4cota mesuresL1debil 9b}) we conclude
\begin{equation}\label{4cota mesuresL1debil 15b}
\begin{split}
\mu\big(\big\{x\in\R^d\setminus\wih\Omega\,:\,IS(x)>\lambda/8\big\}\big)
&\leq\frac{8}{\lambda}\int_{\R^d\setminus\wih\Omega}IS(x)\,d\mu(x)\\
&\leq\frac{8}{\lambda}\sum_j\int_{\R^d\setminus\wih\Omega}\chi_{\R^d\setminus R_j}(x)\bigg|\int K(x-y)\,d\nu_b^j(y)\bigg|\,d\mu(x)\\
&\leq\frac{C}{\lambda}\sum_j|\nu|(Q_j)\leq\frac{C}{\lambda}\|\nu\|.
\end{split}
\end{equation}

Let us estimate $\mu\big(\big\{x\in\R^d\setminus\wih\Omega\,:\,BS(x)>\lambda/8\big\}\big)$. Recall that $\epsilon_m\equiv\epsilon_m(x)$. We define
\begin{equation}\label{4cota mesuresL1debil defi1b}
\begin{split}
\psi^j_m(x)&:=\left\{\begin{array}{ll}
1&\mbox{if $R_j\cap\partial A(x,\epsilon_{m+1}(x),\epsilon_m(x))\neq\emptyset$}\\
0&\mbox{if not}\end{array}\right.,\text{ and}\\
\theta^j_k(x)&:=\left\{\begin{array}{ll}
1&\mbox{if $R_j\cap\partial A(x,2^{-k-1},2^{-k})\neq\emptyset$}\\
0&\mbox{if not}\end{array}\right..
\end{split}
\end{equation}
Then, by the triangle inequality, for $x\in\R^d\setminus\wih\Omega$ we have
\begin{equation}\label{4cota mesuresL1debil 10b}
\begin{split}
BS(x)&=\bigg(\sum_{m\in\Z}\bigg|\sum_{j}\psi^j_m(x)
(K\chi_{\epsilon_{m+1}}^{\epsilon_m}*\nu_b^j)(x)\bigg|^\rho\bigg)^{1/\rho}\\
&\leq\bigg(\sum_{m\in\Z}\bigg|\sum_{j}\chi_{\R^d\setminus 2R_j}(x)\psi^j_m(x)(K\chi_{\epsilon_{m+1}}^{\epsilon_m}*\nu_b^j)(x)\bigg|^\rho\bigg)^{1/\rho}\\
&\quad+\bigg(\sum_{m\in\Z}\bigg|\sum_{j}\chi_{2R_j\setminus 2Q_j}(x)\psi^j_m(x)(K\chi_{\epsilon_{m+1}}^{\epsilon_m}*\nu_b^j)(x)\bigg|^\rho\bigg)^{1/\rho}\\
&\leq\bigg(\sum_{m\in\Z}\bigg|\sum_{j}\chi_{\R^d\setminus 2R_j}(x)\psi^j_m(x)(K\chi_{\epsilon_{m+1}}^{\epsilon_m}*\nu_b^j)(x)\bigg|^\rho\bigg)^{1/\rho}\\
&\quad+\sum_{j}\chi_{2R_j\setminus 2Q_j}(x)\bigg(\sum_{m\in\Z}\big|(K\chi_{\epsilon_{m+1}}^{\epsilon_m}*\nu_b^j)(x)\big|^\rho\bigg)^{1/\rho}\\
&=:BS_1(x)+BS_2(x).
\end{split}
\end{equation}
Notice that $BS_2(x)\leq\sum_{j}\chi_{2R_j\setminus 2Q_j}(x)(\VV_\rho\circ\TT)\nu_b^j(x)$. Since $\rho\geq1$, for $x\in2R_j\setminus2Q_j$,
\begin{equation*}
\begin{split}
(\VV_\rho\circ\TT)\nu_b^j(x)
&\leq(\VV_\rho\circ\TT)(w_j\nu)(x)+(\VV_\rho\circ\TT)(b_j\mu)(x)\\
&\leq(\VV_1\circ\TT)(w_j\nu)(x)+(\VV_\rho\circ\TT^\mu)b_j(x)\\
&\lesssim|\nu|(Q_j)|x-z_j|^{-n}+(\VV_\rho\circ\TT^\mu)b_j(x),
\end{split}
\end{equation*}
where $z_j$ denotes the center of $Q_j$ (and $R_j$).
Then, similarly to (\ref{4cota mesuresL1debil 6b}) and (\ref{4cota mesuresL1debil 7b}) but using now the $L^2(\mu)$ boundedness of $\VV_\rho\circ\TT^\mu$ given by Theorem \ref{4main theorem lip}, we have
\begin{equation}\label{4cota mesuresL1debil 17b}
\begin{split}
\mu\big(\big\{x\in\R^d\setminus&\wih\Omega\,:\,BS_2(x)>\lambda/16\big\}\big)
\leq\frac{16}{\lambda}\int_{\R^d\setminus\wih\Omega}BS_2\,d\mu\\
&\leq\frac{16}{\lambda}\int\sum_{j}\chi_{2R_j\setminus 2Q_j}(\VV_\rho\circ\TT)\nu_b^j\,d\mu
=\frac{16}{\lambda}\sum_{j}\int_{2R_j\setminus 2Q_j}(\VV_\rho\circ\TT)\nu_b^j\,d\mu\\
&\lesssim\frac{1}{\lambda}\sum_j|\nu|(Q_j)\int_{2R_j\setminus2Q_j}|x-z_j|^{-n}\,d\mu(x)
+\frac{1}{\lambda}\sum_{j}\int_{2R_j\setminus 2Q_j}(\VV_\rho\circ\TT^\mu)b_j\,d\mu\\
&\lesssim\frac{1}{\lambda}\sum_j|\nu|(Q_j)\ell(Q_j)^{-n}\mu(2R_j)
+\frac{1}{\lambda}\sum_{j}\|(\VV_\rho\circ\TT^\mu)b_j\|_{L^2(\mu)}\mu(2R_j)^{1/2}\\
&\lesssim\frac{1}{\lambda}\sum_j|\nu|(Q_j)
+\frac{1}{\lambda}\sum_{j}\|b_j\|_{L^\infty(\mu)}\mu(R_j)
\lesssim\frac{1}{\lambda}\sum_j|\nu|(Q_j)\leq\frac{C}{\lambda}\,\|\nu\|.
\end{split}
\end{equation}
Therefore, to show that $\mu\big(\big\{x\in\R^d\setminus\wih\Omega\,:\,BS(x)>\lambda/8\big\}\big)\leq C\lambda^{-1}\|\nu\|$, by (\ref{4cota mesuresL1debil 10b}) and (\ref{4cota mesuresL1debil 17b}) it is enough to verify that
\begin{equation*}
\mu\big(\big\{x\in\R^d\setminus\wih\Omega\,:\,BS_1(x)>\lambda/16\big\}\big)\leq\frac{C}{\lambda}\,\|\nu\|.
\end{equation*}

Without loss of generality, we can assume from the beginning that, for a given $x\in\supp\mu$, either $[\epsilon_{m+1},\epsilon_m)\subset[2^{-k-1},2^{-k})$ for some $k\in\Z$, or $[\epsilon_{m+1},\epsilon_m)=[2^{-i},2^{-k})$ for some $i>k$ (see \cite[page 2130]{CJRW-singular integrals} for a similar argument).
Thus, if we set $I_k:=[2^{-k-1},2^{-k})$, we can decompose $\Z=\SSS\cup\LL$, where
\begin{equation*}
\begin{split}
&\LL:=\{m\in\Z\,:\,\epsilon_m=2^{-k},\,\epsilon_{m+1}=2^{-i}\text{ for }i>k\},\\
&\SSS:=\bigcup_{k\in\Z}\SSS_k,\quad\SSS_k:=\{m\in\Z\,:\,\epsilon_{m},\epsilon_{m+1}\in I_k\}.
\end{split}
\end{equation*}
Then, since $\rho\geq1$,
\begin{equation*}
\begin{split}
BS_1(x)&\leq\bigg(\sum_{m\in\LL}\bigg|\sum_{j}\chi_{\R^d\setminus 2R_j}(x)\psi^j_m(x)
(K\chi_{\epsilon_{m+1}}^{\epsilon_m}*\nu_b^j)(x)\bigg|^\rho\bigg)^{1/\rho}\\
&\quad+\bigg(\sum_{m\in\SSS}\bigg|\sum_{j}\chi_{\R^d\setminus 2R_j}(x)\psi^j_m(x)
(K\chi_{\epsilon_{m+1}}^{\epsilon_m}*\nu_b^j)(x)\bigg|^\rho\bigg)^{1/\rho}\\
&=:BS_\LL(x)+BS_\SSS(x),
\end{split}
\end{equation*}
and we have
\begin{equation}\label{4cota mesuresL1debil 16b}
\begin{split}
\mu\big(\big\{x\in&\R^d\setminus\wih\Omega\,:\,BS_1(x)>\lambda/16\big\}\big)\\
&\leq\mu\big(\big\{x\in\R^d\setminus\wih\Omega\,:\,BS_\LL(x)>\lambda/32\big\}\big)
+\mu\big(\big\{x\in\R^d\setminus\wih\Omega\,:\,BS_\SSS(x)>\lambda/32\big\}\big).
\end{split}
\end{equation}

We are going to estimate first $\mu\big(\big\{x\in\R^d\setminus\wih\Omega\,:\,BS_\LL(x)>\lambda/32\big\}\big)$.
Given $x\in\R^d\setminus\wih\Omega$ (recall the definitions of $\psi^j_k(x)$ and $\theta^j_k(x)$ in (\ref{4cota mesuresL1debil defi1b})), we have
\begin{equation}\label{BSlx}
\begin{split}
BS_\LL(x)&\leq\sum_{j}\sum_{m\in\LL}\chi_{\R^d\setminus 2R_j}(x)\psi^j_m(x)
|(K\chi_{\epsilon_{m+1}}^{\epsilon_m}*\nu_b^j)(x)|\\
&\leq\sum_{j}\sum_{k\in\Z}\chi_{\R^d\setminus 2R_j}(x)\theta^j_k(x)
|(K\chi_{2^{-k-1}}^{2^{-k}}*\nu_b^j)(x)|\\
&\leq\sum_{j}\sum_{k\in\Z\,:\,2^{-k+1}>\ell(R_j)}\chi_{\R^d\setminus 2R_j}(x)\theta^j_k(x)
|(K\chi_{2^{-k-1}}^{2^{-k}}*\nu_b^j)(x)|,
\end{split}
\end{equation}
where in the second and third inequalities above we used the following facts, respectively:
\begin{itemize}
\item assume $m\in\LL$, $\epsilon_{m+1}=2^{-i}$ and
$\epsilon_m=2^{-i+s}$, with $i\in\Z$ and $s\in\N$. Given $j$ such that $R_j\cap\partial A(x,\epsilon_{m+1},\epsilon_m)\neq\emptyset$, if $R_j\cap A(x,2^{-k-1},2^{-k})\neq\emptyset$ for some $k\in\Z$, then $R_j\cap \partial A(x,2^{-k-1},2^{-k})\neq\emptyset$. 
\item For $x\in\R^d\setminus2R_j$, if $2^{-k+1}\leq\ell(R_j)$ then we have $\supp\chi_{2^{-k-1}}^{2^{-k}}(x-\cdot)\cap R_j=\emptyset$, so $(K\chi_{2^{-k-1}}^{2^{-k}}*\nu_b^j)(x)=0$.
\end{itemize}

Therefore, from (\ref{BSlx}) and since $|(K\chi_{2^{-k-1}}^{2^{-k}}*\nu_b^j)(x)|\lesssim2^{(k+1)n}\|\nu_b^j\|$,
\begin{equation}\label{4cota mesuresL1debil 11b}
\begin{split}
\mu\big(\big\{x\in\R^d\setminus\wih\Omega\,:\,&BS_\LL(x)>\lambda/32\big\}\big)
\leq\frac{32}{\lambda}\int_{\R^d\setminus\wih\Omega}BS_\LL(x)\,d\mu(x)\\
&\leq\frac{32}{\lambda}\sum_{j}\sum_{k\in\Z\,:\,2^{-k+1}>\ell(R_j)}\int_{\R^d\setminus 2R_j}\theta^j_k(x)
|(K\chi_{2^{-k-1}}^{2^{-k}}*\nu_b^j)(x)|\,d\mu(x)\\
&\lesssim\frac{1}{\lambda}\sum_{j}\sum_{k\in\Z\,:\,2^{-k+1}>\ell(R_j)}2^{(k+1)n}\|\nu_b^j\|
\int\theta^j_k(x)\,d\mu(x).
\end{split}
\end{equation}

Let us check that
$\int\theta^j_k(x)\,d\mu(x)\lesssim\ell(R_j)2^{-k(n-1)}.$ Fix $k$ and $j$ such that $2^{-k+1}>\ell(R_j)$, and take $u\in\frac{9}{10}R_j\cap\supp\mu$ (this $u$ exists because of (\ref{4lema CZ 2})).
There exists $a>0$ depending only on $d$ such that $\supp\theta_k^j\subset B(u,2^{-k}a)$; thus, if $\ell(R_j)\geq 2^{-k}b$ for some small constant $b>0$, $\int\theta^j_k\,d\mu\leq\mu(B(u,2^{-k}a))\lesssim2^{-kn}\leq b^{-1}\ell(R_j)2^{-k(n-1)}.$ On the contrary, if $\ell(R_j)<2^{-k}b$ and $b$ is small enough, then $$\supp\theta_k^j\subset A(u,2^{-k}-b'\ell(R_j),2^{-k}+b'\ell(R_j))\cup A(u,2^{-k-1}-b'\ell(R_j),2^{-k-1}+b'\ell(R_j))$$ for some constant $b'>0$ depending on $b$ and $d$ such that $2^{-k-1}-b'\ell(R_j)>0$. In that case, since $u\in\supp\mu$, we have $\int\theta^j_k\,d\mu=\mu(\supp\theta_k^j)\lesssim\ell(R_j)2^{-k(n-1)}$ (because $\mu(A(u,r,R))\lesssim(R-r)R^{n-1}$ for all $0<r\leq R$ by Lemma \ref{lema pendent petita3}, since $\Gamma$ has slope smaller than $1$), as desired.

Using that $\int\theta^j_k\,d\mu\lesssim\ell(R_j)2^{-k(n-1)}$ and that $\|\nu_b^j\|\lesssim|\nu|(Q_j)$ in (\ref{4cota mesuresL1debil 11b}), we conclude
\begin{equation}\label{4cota mesuresL1debil 12b}
\begin{split}
\mu\big(\big\{x\in\R^d\setminus\wih\Omega\,:\,BS_\LL(x)>\lambda/32\big\}\big)
&\lesssim\frac{1}{\lambda}\sum_{j}\sum_{k\in\Z\,:\,2^{-k+1}>\ell(R_j)}2^{(k+1)n}\|\nu_b^j\|\ell(R_j)2^{-k(n-1)}\\
&=\frac{1}{\lambda}\sum_{j}\|\nu_b^j\|\ell(R_j)\sum_{k\in\Z\,:\,2^{-k+1}>\ell(R_j)}2^{n+k}\\
&\lesssim\frac{1}{\lambda}\sum_{j}|\nu|(Q_j)\leq\frac{C}{\lambda}\,\|\nu\|.
\end{split}
\end{equation}

It only remains to show $\mu\big(\big\{x\in\R^d\setminus\wih\Omega\,:\,BS_\SSS(x)>\lambda/32\big\}\big)\leq C\lambda^{-1}\|\nu\|$ to finish the proof of the theorem. We set 
$$\Phi_m^j(x):=\chi_{\R^d\setminus 2R_j}(x)\psi^j_m(x)(K\chi_{\epsilon_{m+1}(x)}^{\epsilon_m(x)}*\nu_b^j)(x).$$ 

Recall that $I_r=[2^{-r-1},2^{-r})$. Since the $\ell^\rho$-norm is not bigger than the $\ell^2$-norm, 
\begin{equation*}
\begin{split}
\mu\big(\big\{x\in\R^d\setminus\wih\Omega\,:\,BS_\SSS(x)&>\lambda/32\big\}\big)
\lesssim\frac{1}{\lambda^2}\int_{\R^d\setminus\wih\Omega}\sum_{m\in\SSS}\bigg|\sum_{j}\Phi_m^j(x)\bigg|^2d\mu(x)\\
&=\frac{1}{\lambda^2}\sum_{k\in\Z}\int_{\R^d\setminus\wih\Omega}\sum_{m\in\SSS_k}\bigg|\sum_{j\,:\,2^{-k+1}>\ell(R_j)}
\Phi_m^j(x)\bigg|^2d\mu(x)\\
&=\frac{1}{\lambda^2}\sum_{k\in\Z}\int_{\R^d\setminus\wih\Omega}\sum_{m\in\SSS_k}\bigg|\sum_{r\in\Z\,:\,r\geq k-1}\,\sum_{j\,:\,\ell(R_j)\in I_r} \Phi_m^j(x)\bigg|^2d\mu(x),
\end{split}
\end{equation*}
and then by Cauchy-Schwarz inequality,
\begin{equation*}
\begin{split}
\mu\big(\big\{x\in&\R^d\setminus\wih\Omega\,:\,BS_\SSS(x)>\lambda/32\big\}\big)\\
&\lesssim\frac{1}{\lambda^2}\sum_{k\in\Z}\int_{\R^d\setminus\wih\Omega}\sum_{m\in\SSS_k}
\bigg(\sum_{\begin{subarray}{c} r\in\Z:\\r\geq k-1\end{subarray}}2^{(k-r)/2}\bigg)
\bigg(\sum_{\begin{subarray}{c} r\in\Z:\\r\geq k-1\end{subarray}}2^{(r-k)/2}\bigg|\sum_{j\,:\,\ell(R_j)\in I_r}
\Phi_m^j(x)\bigg|^2\bigg)d\mu(x)\\
&\lesssim\frac{1}{\lambda^2}\sum_{k\in\Z}\int_{\R^d\setminus\wih\Omega}\sum_{m\in\SSS_k}
\sum_{r\in\Z\,:\,r\geq k-1}2^{(r-k)/2}\bigg|\sum_{j\,:\,\ell(R_j)\in I_r}
\Phi_m^j(x)\bigg|^2d\mu(x).
\end{split}
\end{equation*}
Thus, if we set $P_m^r(x):=\sum_{j\,:\,\ell(R_j)\in I_r}\Phi_m^j(x)$, we have seen that 
\begin{equation}\label{4cota mesuresL1debil 13}
\begin{split}
\mu\big(\big\{x\in\R^d\setminus\wih\Omega\,&:\,BS_\SSS(x)>\lambda/32\big\}\big)
\lesssim\frac{1}{\lambda^2}\sum_{k\in\Z}\int_{\R^d\setminus\wih\Omega}\sum_{m\in\SSS_k}
\sum_{\begin{subarray}{c} r\in\Z:\\r\geq k-1\end{subarray}}2^{(r-k)/2}|P_m^r(x)|^2\,d\mu(x).
\end{split}
\end{equation}

Let us estimate $P_m^r(x)$ for $m\in\SSS_k$ and $r\geq k-1$. Since $\|\nu_b^j\|\lesssim|\nu|(Q_j)\leq|\nu|(3Q_j)\lesssim\lambda\mu(6Q_j)$ by (\ref{4lema CZ 5}) and (\ref{4lema CZ 2}), we have
\begin{equation}\label{4cota mesuresL1debil 18}
\begin{split}
|P_m^r(x)|&\leq\sum_{j\,:\,\ell(R_j)\in I_r}
\chi_{\R^d\setminus 2R_j}(x)\psi^j_m(x)|(K\chi_{\epsilon_{m+1}}^{\epsilon_m}*\nu_b^j)(x)|\\
&\lesssim\sum_{j\,:\,\ell(R_j)\in I_r}
\chi_{\R^d\setminus 2R_j}(x)\psi^j_m(x)2^{kn}\|\nu_b^j\|
\lesssim\sum_{\begin{subarray}{c}j\,:\,6\ell(Q_j)\in I_r,\\
6Q_j\cap\partial A(x,\epsilon_{m+1},\epsilon_m)\neq\emptyset\end{subarray}}2^{kn}\lambda\mu(6Q_j).
\end{split}
\end{equation}

It is not difficult to see that, if $\sum_j\chi_{Q_j}\leq C$ for some $C>0$, then 
$\sum_{j\,:\,\ell(6Q_j)\in I_r}\chi_{6Q_j}\leq C'$ for all $r\in\Z$, where $C'>0$ only depends on $C$ (that is, the family of cubes $\FF:=\{6Q_j\}_{j\,:\,\ell(6Q_j)\in I_r}$ has finite overlap uniformly in $r\in\Z$). 
We set 
$$\Upsilon:=\sum_{\begin{subarray}{c}j\,:\,6\ell(Q_j)\in I_r,\\
6Q_j\cap\partial A(x,\epsilon_{m+1},\epsilon_m)\neq\emptyset\end{subarray}}\chi_{6Q_j}.$$
If $2^{-k}a\leq2^{-r}\leq2^{-k+1}$ for some small constant $a>0$ (recall that we are assuming $r\geq k-1$), then there exists a constant $b>0$ depending only on $d$ and $a$ such that
$\supp\Upsilon\subset B(x,b2^{-k}),$
and then, by the finite overlap of the family $\FF$, 
\begin{equation*}
\begin{split}
\sum_{\begin{subarray}{c}j\,:\,6\ell(Q_j)\in I_r,\\
6Q_j\cap\partial A(x,\epsilon_{m+1},\epsilon_m)\neq\emptyset\end{subarray}}\mu(6Q_j)
&=\int_{B(x,b2^{-k})}\Upsilon\,d\mu\leq C'\mu(B(x,b2^{-k}))
\lesssim2^{-kn}\approx2^{-r}2^{-k(n-1)}.
\end{split}
\end{equation*}
On the contrary, if $2^{-k}a\geq2^{-r}$ for $a$ small enough (depending on $d$), then there exists a constant $b>0$ depending on $d$ and $a$ such that $2^{-k-1}>2^{-r}b$ and
$\supp\Upsilon\subset A(x,\epsilon_m-2^{-r}b,\epsilon_m+2^{-r}b)\cup A(x,\epsilon_{m+1}-2^{-r}b,\epsilon_{m+1}+2^{-r}b),$
and then, since $m\in\SSS_k$, $x\in\supp\mu$ and the slope of $\Gamma$ is smaller than $1$, by Lemma \ref{lema pendent petita3}
we have $\mu(\supp\Upsilon)\leq\mu(A(x,\epsilon_m-2^{-r}b,\epsilon_m+2^{-r}b))
+\mu(A(x,\epsilon_{m+1}-2^{-r}b,\epsilon_{m+1}+2^{-r}b))\lesssim2^{-r}2^{-k(n-1)}$, thus by the 
finite overlap of the family $\FF$,
\begin{equation*}
\begin{split}
\sum_{\begin{subarray}{c}j\,:\,6\ell(Q_j)\in I_r,\\
6Q_j\cap\partial A(x,\epsilon_{m+1},\epsilon_m)\neq\emptyset\end{subarray}}\mu(6Q_j)
&=\int_{\supp\Upsilon}\Upsilon\,d\mu
\lesssim\mu(\supp\Upsilon)\lesssim2^{-r}2^{-k(n-1)}.
\end{split}
\end{equation*}

In any case, from (\ref{4cota mesuresL1debil 18}) we get
$|P_m^r(x)|\lesssim2^{kn}\lambda2^{-r}2^{-k(n-1)}=2^{k-r}\lambda.$
Therefore, 
using (\ref{4cota mesuresL1debil 13}) we obtain that
\begin{equation*}
\begin{split}
\mu\big(\big\{x\in&\R^d\setminus\wih\Omega\,:\,BS_\SSS(x)>\lambda/32\big\}\big)
\lesssim\frac{1}{\lambda}\sum_{k\in\Z}\int_{\R^d\setminus\wih\Omega}\sum_{m\in\SSS_k}
\sum_{r\in\Z\,:\,r\geq k-1}2^{(k-r)/2}|P_m^r(x)|\,d\mu(x)\\
&\leq\frac{1}{\lambda}\sum_{k\in\Z}\int_{\R^d\setminus\wih\Omega}\sum_{m\in\SSS_k}
\sum_{\begin{subarray}{c}r\in\Z:\\r\geq k-1\end{subarray}}2^{(k-r)/2}
\sum_{\begin{subarray}{c}j\,:\,\ell(R_j)\in I_r,\\
R_j\cap\partial A(x,\epsilon_{m+1},\epsilon_m)\neq\emptyset\end{subarray}}
|(K\chi_{\epsilon_{m+1}}^{\epsilon_m}*\nu_b^j)(x)|\,d\mu(x)\\
&\lesssim\frac{1}{\lambda}\sum_{k\in\Z}\int_{\R^d\setminus\wih\Omega}\sum_{m\in\SSS_k}
\sum_{\begin{subarray}{c}r\in\Z:\\r\geq k-1\end{subarray}}2^{(k-r)/2}
\sum_{\begin{subarray}{c}j\,:\,\ell(R_j)\in I_r,\\
R_j\cap A(x,2^{-k-1},2^{-k})\neq\emptyset\end{subarray}}
2^{kn}|\nu_b^j|(A(x,\epsilon_{m+1},\epsilon_m))\,d\mu(x)\\
&\leq\frac{1}{\lambda}\sum_{k\in\Z}\int_{\R^d\setminus\wih\Omega}
\sum_{\begin{subarray}{c}r\in\Z:\\r\geq k-1\end{subarray}}2^{(k-r)/2+kn}
\sum_{\begin{subarray}{c}j\,:\,\ell(R_j)\in I_r,\\
R_j\cap A(x,2^{-k-1},2^{-k})\neq\emptyset\end{subarray}}
|\nu_b^j|(A(x,2^{-k-1},2^{-k}))\,d\mu(x).
\end{split}
\end{equation*}
Hence, if we set 
\begin{equation*}
\begin{split}
\tau^j_k(x):=\left\{\begin{array}{ll}
1&\mbox{if $R_j\cap A(x,2^{-k-1},2^{-k})\neq\emptyset$}\\
0&\mbox{if not}\end{array}\right.,
\end{split}
\end{equation*}
we obtain
\begin{equation*}
\begin{split}
\mu\big(\big\{x\in\R^d\setminus\wih\Omega\,&:\,BS_\SSS(x)>\lambda/32\big\}\big)\\
&\leq\frac{1}{\lambda}\sum_{k\in\Z}\int_{\R^d\setminus\wih\Omega}
\sum_{\begin{subarray}{c}r\in\Z:\\r\geq k-1\end{subarray}}2^{(k-r)/2+kn}
\sum_{j\,:\,\ell(R_j)\in I_r}\|\nu_b^j\|\tau_k^j(x)\,d\mu(x)\\
&=\frac{1}{\lambda}\sum_{k\in\Z}\,
\sum_{\begin{subarray}{c}r\in\Z:\\r\geq k-1\end{subarray}}2^{(k-r)/2+kn}
\sum_{j\,:\,\ell(R_j)\in I_r}\|\nu_b^j\|\int_{\R^d\setminus\wih\Omega}\tau_k^j\,d\mu.
\end{split}
\end{equation*}
Notice that, if $\ell(R_j)\in I_r$ and $r\geq k-1$, then $\ell(R_j)<2^{-k+1}$. Hence, there exists a constant $C>0$ such that $\supp\tau^j_k\subset B(z_j,C2^{-k})$ for all $\ell(R_j)\in I_r$ and all $r\geq k-1$ (recall that $z_j$ is the center of $R_j$), and then 
$\int_{\R^d\setminus\wih\Omega}\tau_k^j\,d\mu\leq\mu(B(z_j,C2^{-k}))\lesssim2^{-kn}$. Therefore, by exchanging the order of summation and using that $\|\nu_b^j\|\lesssim|\nu|(Q_j)$, we finally obtain
\begin{equation}\label{4cota mesuresL1debil 14b}
\begin{split}
\mu\big(\big\{x\in\R^d\setminus&\wih\Omega\,:\,BS_\SSS(x)>\lambda/32\big\}\big)
\lesssim\frac{1}{\lambda}\sum_{k\in\Z}\,\sum_{r\in\Z\,:\,r\geq k-1}2^{(k-r)/2}
\sum_{j\,:\,\ell(R_j)\in I_r}\|\nu_b^j\|\\
&=\frac{1}{\lambda}\sum_{j}|\nu|(Q_j)\sum_{r\in\Z\,:\,2^{-r-1}\leq\ell(R_j)<2^{-r}}\,\sum_{k\in\Z\,:\,k\leq r+1}\,2^{(k-r)/2}\\
&\lesssim\frac{1}{\lambda}\sum_{j}|\nu|(Q_j)\leq\frac{C}{\lambda}\|\nu\|.
\end{split}
\end{equation}

The estimate
(\ref{4cota mesuresL1debil 1}) 
is a direct consequence of (\ref{4cota mesuresL1debil 5b}), (\ref{4cota mesuresL1debil 15b}), (\ref{4cota mesuresL1debil 10b}), (\ref{4cota mesuresL1debil 17b}), (\ref{4cota mesuresL1debil 16b}), (\ref{4cota mesuresL1debil 12b}), and (\ref{4cota mesuresL1debil 14b}).
\end{proof}

\section{$\VV_\rho\circ\TT^{\HH^n_\Gamma}$ is a bounded operator from $L^\infty(\HH^{n}_{\Gamma})$ to $BMO(\HH^{n}_{\Gamma})$} \label{teorema Lp no suau s2}

This section is devoted to the proof of the endpoint estimate $(c)$ of Theorem \ref{4unif rectif teo3b}. The use of Lemma \ref{lema pendent petita3} is also essential in this section.

We may assume that $\Gamma=\{(y,\A(y)):\, y\in\R^n\},$ where $\A:\R^n\to\R^{d-n}$ is some Lipschitz function with Lipschitz constant $\Lip(\A)$. We say that a function $f\in L^1_{loc}(\HH^n_\Gamma)$ belongs to $BMO(\HH^n_\Gamma)$ if there exists a constant $C>0$ such that $$\sup_D\inf_{c\in\R}\frac{1}{\HH^n_\Gamma(D)}\int_D|f-c|\,d\HH^n_\Gamma\leq C,$$
where the supremum is taken over all the sets of the type
$D:=\wit D\times\R^{d-n}$, where $\wit D$ is a cube in $\R^n$. For convenience of notation, given $a>0$ we define $aD:=a\wit D\times\R^{d-n}$ and $\ell(aD):=\ell(a\wit D)$. Notice that, since $\Gamma$ is an $n$-dimensional Lipschitz graph, we have $\HH^n_\Gamma(D)\approx\ell(D)^n$ for all cubes $\wit D\subset\R^n$. Moreover $(\Gamma,\HH^n_\Gamma)$ is a space of homogeneous type, and it is not hard to show that our definition of $BMO(\HH^n_\Gamma)$ is equivalent to the classical one for doubling measures (see \cite{Tolsa9} for a definition of $BMO$ on doubling measures).

\begin{proof}[{\bf {\em Proof of} Theorem \ref{4unif rectif teo3b}$(c)$}]
We have to prove that there exists a constant $C>0$ such that, for any $f\in L^\infty(\HH^{n}_{\Gamma})$ and any cube $\widetilde D\subset\R^n$, there exists some constant $c$ depending on $f$ and $\wit D$ such that 
\begin{equation}\label{lala2}
\begin{split}
\int_{\wit D\times\R^{d-n}}\big|(\VV_\rho\circ\TT^{\HH^{n}_{\Gamma}})f-c\big|\,d\HH^{n}_{\Gamma}\leq C\|f\|_{L^\infty(\HH^n_\Gamma)}\HH^n_\Gamma(\wit D\times\R^{d-n}).
\end{split}
\end{equation}

Let $f$ and $\wit D$ be as above, and set $D:=\wit D\times\R^{d-n}$, $f_1:=f\chi_{3D}$, and $f_2:=f-f_1$.
First of all, by H\"{o}lder's inequality, Theorem \ref{4main theorem lip}, and since $\HH^{n}_{\Gamma}(3D)\approx\HH^{n}_{\Gamma}(D)$ because $\Gamma$ is a Lipschitz graph,  we have
\begin{equation}\label{lala}
\begin{split}
\int_{D}(\VV_\rho\circ\TT^{\HH^{n}_{\Gamma}})f_1\,d\HH^{n}_{\Gamma}&\leq\HH^{n}_{\Gamma}(D)^{1/2}
\bigg(\int\big((\VV_\rho\circ\TT^{\HH^{n}_{\Gamma}})f_1\big)^2\,d\HH^{n}_{\Gamma}\bigg)^{1/2}\\
&\lesssim\HH^{n}_{\Gamma}(D)^{1/2}\left(\|f_1\|^2_{L^\infty(\HH^n_\Gamma)}\HH^{n}_{\Gamma}(3D)\right)^{1/2}\lesssim\|f\|_{L^\infty(\HH^n_\Gamma)}\HH^{n}_{\Gamma}(D).
\end{split}
\end{equation}

Notice that $|(\VV_\rho\circ\TT^{\HH^{n}_{\Gamma}})(f_1+f_2)-(\VV_\rho\circ\TT^{\HH^{n}_{\Gamma}})f_2|\leq
(\VV_\rho\circ\TT^{\HH^{n}_{\Gamma}})f_1$, because $\VV_\rho\circ\TT^{\HH^{n}_{\Gamma}}$ is sublinear and positive. Then, for any $c\in\R$,
\begin{equation}\label{lala1}
\begin{split}
|(\VV_\rho\circ\TT^{\HH^{n}_{\Gamma}})f-c|&=|(\VV_\rho\circ\TT^{\HH^{n}_{\Gamma}})(f_1+f_2)-c|\\
&\leq|(\VV_\rho\circ\TT^{\HH^{n}_{\Gamma}})(f_1+f_2)-(\VV_\rho\circ\TT^{\HH^{n}_{\Gamma}})f_2|
+|(\VV_\rho\circ\TT^{\HH^{n}_{\Gamma}})f_2-c|\\
&\leq(\VV_\rho\circ\TT^{\HH^{n}_{\Gamma}})f_1
+|(\VV_\rho\circ\TT^{\HH^{n}_{\Gamma}})f_2-c|,
\end{split}
\end{equation}
hence, to prove (\ref{lala2}), by (\ref{lala}) and (\ref{lala1}) we are reduced to prove that, for some constant $c\in\R$,
\begin{equation}\label{teorema interpolacio Linf1}
\int_D\big|(\VV_\rho\circ\TT^{\HH^{n}_{\Gamma}})f_2-c\big|\,d\HH^{n}_{\Gamma}\leq C\|f\|_{L^\infty(\HH^n_\Gamma)}\HH^n_\Gamma(D).
\end{equation}

Set $z_D:=(\widetilde z_D,\A(\widetilde z_D))$, where $\wit z_D$ is the center of $\wit D\subset\R^n$, and take
$c:=(\VV_\rho\circ\TT^{\HH^{n}_{\Gamma}})f_2(z_D)$. We may assume that $c<\infty$. By the triangle inequality,
\begin{equation*}
\big|(\VV_\rho\circ\TT^{\HH^{n}_{\Gamma}})f_2(x)-c\,\big|^\rho
\leq \sup_{\{\epsilon_m\searrow0\}}\sum_{m\in\Z}
|(K\chi_{\epsilon_{m+1}}^{\,\epsilon_m}*(f_2\HH^n_\Gamma))(x)-
(K\chi_{\epsilon_{m+1}}^{\,\epsilon_m}*(f_2\HH^n_\Gamma))(z_D)|^{\rho}.
\end{equation*}
Given $x\in\Gamma\cap D$, let $\{\epsilon_m\}_{m\in\Z}$ be a decreasing sequence of positive numbers (which depends on $x$) such that
\begin{equation*}
\big|(\VV_\rho\circ\TT^{\HH^{n}_{\Gamma}})f_2(x)-c\,\big|^\rho\\
\leq 2\sum_{m\in\Z}|(K\chi_{\epsilon_{m+1}}^{\,\epsilon_m}*(f_2\HH^n_\Gamma))(x)-
(K\chi_{\epsilon_{m+1}}^{\,\epsilon_m}*(f_2\HH^n_\Gamma))(z_D)|^{\rho}.
\end{equation*}
Notice that $|(K\chi_{\epsilon_{m+1}}^{\,\epsilon_m}*(f_2\HH^n_\Gamma))(x)-
(K\chi_{\epsilon_{m+1}}^{\,\epsilon_m}*(f_2\HH^n_\Gamma))(z_D)|\leq\|f\|_{L^\infty(\HH^n_\Gamma)}(\Theta1_m+\Theta2_m)$, where
\begin{equation*}
\begin{split}
\Theta1_m:&=\int_{(3D)^c}\chi_{\epsilon_{m+1}}^{\,\epsilon_m}(x-y)\left|K(x-y)-K(z_D-y)\right|\,d\HH^n_\Gamma(y),\\
\Theta2_m:&=\int_{(3D)^c}|\chi_{\epsilon_{m+1}}^{\,\epsilon_m}(x-y)-\chi_{\epsilon_{m+1}}^{\,\epsilon_m}(z_D-y)||K(z_D-y)|\,d\HH^n_\Gamma(y).
\end{split}
\end{equation*}
Thus, 
\begin{equation}\label{sasa}
\big|(\VV_\rho\circ\TT^{\HH^{n}_{\Gamma}})f_2(x)-c\,\big|
\lesssim\|f\|_{L^\infty(\HH^n_\Gamma)}\bigg(\sum_{m\in\Z}(\Theta1_m+\Theta2_m)^\rho\bigg)^{1/\rho}.
\end{equation}

Since $\rho\geq 1$, we easily have
\begin{equation}\label{sasa1}
\begin{split}
\bigg(\sum_{m\in\Z}\Theta1_m^\rho\bigg)^{1/\rho}&\leq\sum_{m\in\Z}\Theta1_m
\lesssim\int_{(3D)^c}\sum_{m\in\Z}\chi_{\epsilon_{m+1}}^{\,\epsilon_m}(x-y)\frac{|x-z_D|}{|z_D-y|^{n+1}}
\,d\HH^n_\Gamma(y)\\
&\lesssim\ell(D)\int_{(3D)^c}|z_D-y|^{-n-1}\,d\HH^n_\Gamma(y)\lesssim1.
\end{split}
\end{equation}

The case of $\Theta2_m$ is more delicate. Since $\Gamma$ is a Lipschitz graph, there exists an integer $M>10$ depending only on $\Lip(\A)$ such that any $x\in\Gamma\cap D$ satisfies $|x-z_D|<2^M\ell(D)$. Without loss of generality, we can assume that there exists $m_0\in\Z$ such that $\epsilon_{m_0}=2^{M+2}\ell(D)$, just by adding the term $2^{M+2}\ell(D)$ to the fixed sequence $\{\epsilon_m\}_{m\in\Z}$. Obviously, we can also assume that $\epsilon_m>\epsilon_{m+1}$ for all $m\in\Z$.

We set $J_0:=\{m\in\Z\,:\,\epsilon_m\leq2^{M+2}\ell(D)\}=\{m\in\Z\,:\,m\geq m_0\}$ and, for $j>M+2$,
\begin{align*}
&J^1_j:=\{m\in\Z\,:\,2^{j-1}\ell(D)\leq\epsilon_{m+1}<\epsilon_m\leq2^{j}\ell(D)\text{ and }\epsilon_m-\epsilon_{m+1}\geq2^M\ell(D)\},\\
&J^2_j:=\{m\in\Z\,:\,2^{j-1}\ell(D)\leq\epsilon_{m+1}<\epsilon_m\leq2^{j}\ell(D)\text{ and }\epsilon_m-\epsilon_{m+1}<2^M\ell(D)\},\\
&J^3_j:=\{m\in\Z\,:\,2^{j-1}\ell(D)\leq\epsilon_{m+1}\leq2^{j}\ell(D)<\epsilon_m\}.
\end{align*}

Then $\Z=J_0\cup\big(\bigcup_{j>M+2}(J_j^1\cup J_j^2\cup J_j^3)\big)$. For the case of $m\in J_0$, we have the easy estimate
\begin{equation*}
\begin{split}
\bigg(\sum_{m\in J_0}\Theta2_m^\rho\bigg)^{1/\rho}
&\lesssim\sum_{m\in J_0}\int_{(3D)^c}\big(\chi_{\epsilon_{m+1}}^{\,\epsilon_m}(x-y)+\chi_{\epsilon_{m+1}}^{\,\epsilon_m}(z_D-y)
\big)\ell(D)^{-n}\,d\HH^n_\Gamma(y)\\
&\leq\int_{|x-y|\leq2^{M+2}\ell(D)}\frac{\,d\HH^n_\Gamma(y)}{\ell(D)^{n}}
+\int_{|z_D-y|\leq2^{M+2}\ell(D)}\frac{\,d\HH^n_\Gamma(y)}{\ell(D)^{n}}\lesssim1.
\end{split}
\end{equation*}

Assume that $m\in J^1_j$. Notice that $\supp\big(\chi_{\epsilon_{m+1}}^{\,\epsilon_m}(x-\cdot)-\chi_{\epsilon_{m+1}}^{\,\epsilon_m}(z_D-\cdot)\big)\subset A_m(x,z_D)$, where $A_m(x,z_D)$ denotes the symmetric difference between $A(x,\epsilon_{m+1},\epsilon_m)$ and $A(z_D,\epsilon_{m+1},\epsilon_m)$. Notice also that, since $m\in J_j^1$ and $x\in D\cap\Gamma$, the set $A_m(x,z_D)$ is contained in the union of annuli $A_1:=A(x,\epsilon_{m+1}-2^M\ell(D),\epsilon_{m+1}+2^M\ell(D))$ and $A_2:=A(x,\epsilon_m-2^M\ell(D),\epsilon_m+2^M\ell(D)).$ Hence, using that $m\in J^1_j$ and Lemma \ref{lema pendent petita3}, we have
\begin{equation}\label{teorema Lp no suau 1}
\begin{split}
\HH^n_\Gamma(\{y\in\R^d\,:\,|&\chi_{\epsilon_{m+1}}^{\epsilon_m}(x-y)-\chi_{\epsilon_{m+1}}^{\epsilon_m}(z_D-y)|\neq0\})
\leq\HH^n_\Gamma(A_1\cup A_2)\\
&\lesssim 2^{M+1}\ell(D)\Big(\epsilon_m+2^M\ell(D)\Big)^{n-1}
+2^{M+1}\ell(D)\Big(\epsilon_{m+1}+2^M\ell(D)\Big)^{n-1}\\
&\lesssim2^{j(n-1)}\ell(D)^n.
\end{split}
\end{equation}
Using that $|K(z_D-y)|\lesssim(2^{j}\ell(D))^{-n}$ for all $y\in A_m(x,z_D)\cap(3D)^c$, we get $$\Theta2_m\lesssim(2^{j}\ell(D))^{-n}2^{j(n-1)}\ell(D)^n=2^{-j}$$ and, since $\rho\geq2$ and $J^1_j$ contains at most $2^{j-M-1}$ indices, we have $\sum_{m\in J^1_j}\Theta2_m^\rho\lesssim2^{-j}$.

Assume now that $m\in J^2_j$. Then, using Lemma \ref{lema pendent petita3}, we obtain
\begin{equation*}
\begin{split}
\HH^n_\Gamma(\{y\in\R^d\,&:\,|\chi_{\epsilon_{m+1}}^{\,\epsilon_m}(x-y)-\chi_{\epsilon_{m+1}}^{\,\epsilon_m}(z_D-y)|\neq0\})\\
&\leq\HH^n_\Gamma(\{y\in\R^d\,:\,\chi_{\epsilon_{m+1}}^{\,\epsilon_m}(x-y)=1\})
+\HH^n_\Gamma(\{y\in\R^d\,:\,\chi_{\epsilon_{m+1}}^{\,\epsilon_m}(z_D-y)=1\})\\
&\lesssim(\epsilon_m-\epsilon_{m+1})\epsilon_m^{n-1},
\end{split}
\end{equation*}
and, as above, $|K(z_D-y)|\lesssim(2^{j}\ell(D))^{-n}$ for all $y\in A_m(x,z_D)\cap(3D)^c$. Since $m\in J^2_j$,
\begin{equation*}
\begin{split}
\Theta2^\rho_m&\lesssim(2^{j}\ell(D))^{-\rho n}\big((\epsilon_m-\epsilon_{m+1})\epsilon_m^{n-1}\big)^{\rho}\\
&\lesssim(2^{j}\ell(D))^{-\rho n}(\epsilon_m-\epsilon_{m+1})(2^M\ell(D))^{\rho-1}(2^j\ell(D))^{(n-1)\rho}\lesssim2^{-j\rho}\ell(D)^{-1}(\epsilon_m-\epsilon_{m+1})
\end{split}
\end{equation*}
and then, since $\rho\geq2$ and $j>M+2>12$,
\begin{equation*}
\begin{split}
\sum_{m\in J^2_j}\Theta2^\rho_m&\lesssim2^{-j\rho}\sum_{m\in J^2_j}\frac{\epsilon_m-\epsilon_{m+1}}{\ell(D)}\leq2^{-j\rho}2^{j-1}\approx2^{-j(\rho-1)}\leq2^{-j}.
\end{split}
\end{equation*}

Finally, assume that $m\in J^3_j$. Obviously, the set $J^3_j$ contains at most one term.
If $\epsilon_m-\epsilon_{m+1}<2^M\ell(D)$, arguing as in the case $m\in J_j^2$, we have
\begin{equation*}
\begin{split}
\HH^n_\Gamma(\{y\in\R^d\,:\,|\chi_{\epsilon_{m+1}}^{\,\epsilon_m}&(x-y)-\chi_{\epsilon_{m+1}}^{\,\epsilon_m}(z_D-y)|\neq0\})\lesssim(\epsilon_m-\epsilon_{m+1})\epsilon_m^{n-1}\\
&\lesssim2^M\ell(D)(2^j\ell(D)+2^M\ell(D))^{n-1}\lesssim2^{j(n-1)}\ell(D)^n,
\end{split}
\end{equation*}
and then $\Theta2_m\lesssim2^{j(n-1)}\ell(D)^n(2^{j-1}\ell(D))^{-n}\lesssim2^{-j}$.
On the contrary, if $\epsilon_m-\epsilon_{m+1}\geq2^M\ell(D)$, arguing as in the case $m\in J_j^1$, we have
$\supp\big(\chi_{\epsilon_{m+1}}^{\,\epsilon_m}(x-\cdot)-\chi_{\epsilon_{m+1}}^{\,\epsilon_m}(z_D-\cdot)\big)\subset A_m(x,z_D)\subset A_1\cup A_2$. Similarly to (\ref{teorema Lp no suau 1}), we have
$$\HH^n_\Gamma(A_1)\lesssim 2^{M+1}\ell(D)(\epsilon_{m+1}+2^{M}\ell(D))^{n-1}\lesssim\epsilon_{m+1}^{n-1}\ell(D)\leq2^{j(n-1)}\ell(D)^n,$$ and $|K(z_D-y)|\lesssim(2^{j}\ell(D))^{-n}$ for all $y\in A_1\cap(3D)^c$. If we denote by $j(\epsilon_{m})$ the positive integer such that $2^{j(\epsilon_m)-1}\ell(D)\leq\epsilon_m\leq2^{j(\epsilon_m)}\ell(D)$ (obviously, $j(\epsilon_m)>j$), we have
$\HH^n_\Gamma(A_2)\lesssim\epsilon_m^{n-1}\ell(D)\leq2^{j(\epsilon_m)(n-1)}\ell(D)^n,$ and $|K(z_D-y)|\lesssim(2^{j(\epsilon_m)}\ell(D))^{-n}$ for all $y\in A_2\cap(3D)^c$. Hence,
$\Theta2_m\lesssim2^{j(n-1)}\ell(D)^n(2^{j}\ell(D))^{-n}
+2^{j(\epsilon_m)(n-1)}\ell(D)^n(2^{j(\epsilon_m)}\ell(D))^{-n}\lesssim2^{-j}+2^{-j(\epsilon_m)}\lesssim2^{-j}$.
Therefore, since $J^3_j$ contains at most one term, $\sum_{m\in J^3_j}\Theta2^\rho_m\lesssim2^{-j\rho}\leq2^{-j}$.

We put all these estimates of $\Theta2_m$ for $m$ belonging to $J_0$, $J_j^1$, $J_j^2$, and $J_j^3$ together with (\ref{sasa1}) in (\ref{sasa}) and we conclude that
\begin{equation*}
\begin{split}
\big|(\VV_\rho\circ\TT^{\HH^{n}_{\Gamma}})f_2(x)-c\,\big|
&\lesssim\|f\|_{L^\infty(\HH^n_\Gamma)}\bigg(\sum_{m\in\Z}(\Theta1_m+\Theta2_m)^\rho\bigg)^{1/\rho}\\
&\lesssim\|f\|_{L^\infty(\HH^n_\Gamma)}\bigg(\sum_{m\in\Z}\Theta1^\rho_m\bigg)^{1/\rho}
+\|f\|_{L^\infty(\HH^n_\Gamma)}\bigg(\sum_{m\in J_0}\Theta2^\rho_m\bigg)^{1/\rho}\\
&\quad+\|f\|_{L^\infty(\HH^n_\Gamma)}\bigg(\sum_{j>M+2}\bigg(\sum_{m\in J^1_j}\Theta2^\rho_m+\sum_{m\in J^2_j}\Theta2^\rho_m+\sum_{m\in J^3_j}\Theta2^\rho_m\bigg)\bigg)^{1/\rho}\\
&\lesssim\|f\|_{L^\infty(\HH^n_\Gamma)}\bigg(1
+1+\bigg(\sum_{j>12}2^{-j}\bigg)^{1/\rho}\bigg)\lesssim\|f\|_{L^\infty(\HH^n_\Gamma)}.
\end{split}
\end{equation*}
Finally, (\ref{teorema interpolacio Linf1}) follows by integrating in $D$ this last estimate. This yields the boundedness of $\VV_\rho\circ\TT^{\HH^{n}_{\Gamma}}$ from $L^\infty(\HH^n_\Gamma)$ to $BMO(\HH^n_\Gamma)$.
\end{proof}

\section{$\VV_\rho\circ\TT^{\HH^n_\Gamma}$ is a bounded operator in $L^p(\HH^{n}_{\Gamma})$ for all $1<p<\infty$}\label{s endsec}
This section is devoted to complete the proof of Theorem \ref{4unif rectif teo3b} and Corollary \ref{coro pppp}.

\begin{proof}[{\bf {\em Proof of} Theorem \ref{4unif rectif teo3b}$(b)$}]
This is a straightforward application of Theorem \ref{4unif rectif teo3}.
\end{proof}
\begin{proof}[{\bf {\em Proof of} Theorem \ref{4unif rectif teo3b}$(a)$}]
Recall from Theorem \ref{4main theorem lip} that $\VV_\rho\circ\TT^{\HH^{n}_{\Gamma}}$ is bounded in $L^2(\HH^{n}_{\Gamma})$. We deduce the $L^p$ boundedness of the positive sublinear operator
$\VV_\rho\circ\TT^{\HH^{n}_{\Gamma}}$ by interpolation between the
pairs $(L^1(\HH^n_\Gamma),L^{1,\infty}(\HH^n_\Gamma))$  and
$(L^2(\HH^n_\Gamma),L^2(\HH^n_\Gamma))$  for $1<p<2$, and between
$(L^2(\HH^n_\Gamma),L^2(\HH^n_\Gamma))$  and
$(L^\infty(\HH^n_\Gamma),BMO(\HH^n_\Gamma))$ for $2<p<\infty$.
Let us remark that, in the latter case, the classical interpolation
theorem (see \cite[Theorem 6.8]{Duoandi}, for instance) would
require the operator $\VV_\rho\circ\TT^{\HH^{n}_{\Gamma}}$ to be
linear. Clearly, this fails in our case. However, an easy
modification of the arguments in \cite{Duoandi} using Lemma
\ref{lema interpolacio} below shows that that interpolation theorem is also valid for positive sublinear operators. Before stating the lemma, 
let us recall some definitions. Given $f\in L^1_{loc}(\HH^n_\Gamma)$, $x\in\R^d$, and a cube $\wit Q\in\R^n$, set $Q=\wit Q\times\R^{d-n}$ and define 
$$m_Q f:=\frac{1}{\HH^n_\Gamma(Q)}\int_Q f\,d\HH^n_\Gamma,$$
$Mf(x):=\sup_{Q\ni x}m_Q |f|$, and 
$M^\sharp f(x):=\sup_{Q\ni x}m_Q|f-m_Qf|$.

\begin{lema}\label{lema interpolacio}
Let $F:L^1_{loc}(\HH^{n}_{\Gamma})\to L^1_{loc}(\HH^{n}_{\Gamma})$ be a positive and sublinear operator. Then $(M^\sharp\circ F)(f+g)\lesssim (M\circ F)f+(M^\sharp\circ F)g$ for all functions $f,g\in L^1_{loc}(\HH^{n}_{\Gamma})$.
\end{lema}

By using Lemma \ref{lema interpolacio} and the fact that $\|M f\|_{L^p(\HH^{n}_{\Gamma})}\lesssim\|M^\sharp f\|_{L^p(\HH^{n}_{\Gamma})}$ for $f\in L^{p_0}(\HH^{n}_{\Gamma})$ and $1\leq p_0\leq p<\infty$ (see \cite[Lemma 6.9]{Duoandi}), one can reprove the interpolation theorem \cite[Theorem 6.8]{Duoandi} applied to $\VV_\rho\circ\TT^{\HH^{n}_{\Gamma}}$ with minor modifications in the original proof.
\end{proof}

\begin{proof}[{\bf {\em Proof of} Lemma \ref{lema interpolacio}}]
If $F$ is sublinear and positive, one has that $|F(f)(x)-F(g)(x)|\leq F(f-g)(x)$ for all functions $f,g\in L^1_{loc}(\HH^{n}_{\Gamma})$. Let $\wit Q$ be a cube in $\R^n$, and set $Q=\wit Q\times\R^{d-n}\subset\R^d$. Then, for $x,y\in Q\cap\Gamma$,
\begin{equation*}
\begin{split}
|F(f+g)(y)-m_Q(Fg)|&\leq|F(f+g)(y)-Fg(y)|+|Fg(y)-m_Q(Fg)|\\
&\leq|Ff(y)|+|Fg(y)-m_Q(Fg)|.
\end{split}
\end{equation*}
Hence, $m_Q|F(f+g)-m_Q(Fg)|\leq m_Q|Ff|+m_Q|Fg-m_Q(Fg)|\leq (M\circ F)f(x)+(M^\sharp\circ F)g(x)$ and, by taking the supremum over all possible cubes $\wit Q\subset\R^n$ such that $Q\ni x$, we conclude $(M^\sharp\circ F)(f+g)(x)\lesssim(M\circ F)f(x)+ (M^\sharp\circ F)g(x)$ (recall that $(M^\sharp\circ F)h(x)\lesssim\sup_{Q\ni x}\inf_{a\in\R}m_Q|Fh-a|$ for all $h\in L^1_{loc}(\HH^{n}_{\Gamma})$).
\end{proof}

\begin{proof}[{\bf {\em Proof of} Corollary \ref{coro pppp}}]
The arguments follow closely the proof of \cite[Theorem 20.27]{Mattila-llibre}. First of all, we may assume that $E$ is a Lipschitz graph with slope smaller than 1, since $\HH^n$ almost all $E$ can be covered with countably many $\CC^1$ manifolds which in turn can be covered by Lipschitz graphs with small slope. By the Lebesgue decomposition theorem and Radon-Nikodym theorem (see \cite[Theorem 2.17]{Mattila-llibre} for the real case, for example), there exists $f\in L^1(\HH^n_E)$ and a finite complex Radon measure $\nu_s$ such that $\HH^n_E$ and $|\nu_s|$ are mutually singular and $\nu=f\HH^n_E+\nu_s$.

Given $K$ satisfying (\ref{4eq333}), by Theorem \ref{4unif rectif teo3b}$(b)$ we have $(\VV_{\rho}\circ\TT^{\HH^n_E})f(x)<\infty$ for $\HH^{n}$ almost all $x\in E$. Therefore, for any decreasing sequence $\{\epsilon_m\}_{m\in\Z}$, $\{T_{\epsilon_m}^{\mu}f(x)\}_{m\in\Z}$ is a Cauchy sequence, so it is convergent. Thus $\lim_{\epsilon\to0}T^{\HH^n_E}_\epsilon f(x)$ exists for $\HH^{n}$ almost all $x\in E$. Therefore, we may assume that $\nu=\nu_s$. The rest of the proof is almost the same of \cite[Theorem 20.27]{Mattila-llibre}  (just replace $T^*$ by $\VV_\rho\circ\TT$ in the proof in \cite{Mattila-llibre} and use Theorem \ref{4unif rectif teo3}). The details are left for the reader.
\end{proof}

\end{document}